\documentclass[a4paper,10pt]{article}
\usepackage[latin1]{inputenc}
\usepackage[T1]{fontenc}
\usepackage{amsmath}
\usepackage{enumerate}
\usepackage{amsfonts}
\usepackage{amstext}
\usepackage{amsthm}
\usepackage[all]{xy}
\usepackage{geometry}
\usepackage{srcltx}
\usepackage{graphics}
\usepackage{epsfig}
\usepackage{subfigure}
\usepackage{slashed}
\usepackage{color}
\usepackage{bbm}
\usepackage{amssymb}
\usepackage{sidecap}
\usepackage{srcltx}
\usepackage{comment}

\DeclareMathOperator{\sgn}{sgn}
\DeclareMathOperator{\spann}{span}
\DeclareMathOperator{\id}{id}
\DeclareMathOperator{\ind}{ind}

\DeclareMathOperator{\codim}{codim}

\DeclareMathOperator{\GL}{GL}

\DeclareMathOperator{\RE}{Re}
\DeclareMathOperator{\KO}{KO}

\DeclareMathOperator{\Mat}{Mat}

\newcommand{\R}{{\mathbb R}}
\newcommand{\Z}{{\mathbb Z}}

\newcommand{\cD}{{\mathcal D}}

\renewcommand{\d}{\,{\mathrm d}}
\newcommand{\tm}{\times}

\newcommand{\set}[1]{\left\{#1\right\}}		
\newcommand{\fall}{\;\text{ for all }}		
\newtheorem{theorem}{Theorem}[section]
\newtheorem{lemma}[theorem]{Lemma}
\newtheorem{corollary}[theorem]{Corollary}
\newtheorem{proposition}[theorem]{Proposition}
\newtheorem*{theorem*}{Theorem}

\theoremstyle{definition}
\newtheorem{definition}[theorem]{Definition}

\newtheorem*{hypothesis*}{Hypothesis}
\theoremstyle{remark}

\newcommand{\cref}[1]{Cor.~\ref{#1}}

\newcommand{\href}[1]{Hyp.~\ref{#1}}

\title{An Index Theorem in Relative K-Theory for First-Order Systems}
\author{Robert Skiba, Daniel Strzelecki and Nils Waterstraat}

\begin{document}
\date{}
\maketitle


\begin{abstract}
Motivated by bifurcation of branches of homoclinic orbits of dynamical systems, we consider families of first-order equations on the real line and introduce a generalisation of previous index theorems by Pejsachowicz, and by Hu and Portaluri. The main novelties of our approach firstly concern the analytical setting, where we lift the common assumption that the equations are asymptotically hyperbolic. Secondly, we consider general compact parameter spaces instead of a single parameter, which results in a remarkably simple index formula in relative $K$-theory.
\end{abstract}

\section{Introduction}
The motivation of this paper stems from bifurcation of homoclinic solutions of general nonlinear systems of the form
\begin{equation}\label{Hamilton}
\left\{
\begin{array}{l}
\dot{u}(t)= g(\lambda,t,u(t)), \\
\lim\limits_{t\to\pm\infty} u(t)=0,
\end{array}
\right.
\end{equation}
where $g:[a,b]\times\mathbb{R}\times\mathbb{R}^d\rightarrow\mathbb{R}^d$ is continuously differentiable and $u\equiv 0$ satisfies \eqref{Hamilton} for all $\lambda\in [a,b]$. To study bifurcation from this trivial solution, the linearised systems
\begin{equation}\label{lin-Hamilton}
\left\{
\begin{array}{l}
\dot{u}(t)= A_\lambda(t)u(t), \\
\lim\limits_{t\to\pm\infty} u(t)=0
\end{array}
\right.
\end{equation}
for $A_\lambda(t):=D_ug(\lambda,t,0)$ play a crucial role as under common assumptions a non-trivial solution space of \eqref{lin-Hamilton} at some parameter $\lambda^\ast\in(a,b)$ is necessary for new solutions of \eqref{Hamilton} to emerge out of the trivial one when $\lambda$ passes $\lambda^\ast$. Finding sufficient criteria for a bifurcation from $u\equiv 0$ is a far more sophisticated problem, in particular as the classical Krasnoselskii bifurcation theorems are not applicable due to the lack of compactness.\\
Note that \eqref{lin-Hamilton} has a non-trivial solution if and only if
\[E^u_\lambda(0)\cap E^s_\lambda(0)\neq\{0\},\]
where
\begin{align}\label{stableunstable}
\begin{split}
E^s_\lambda(0)&=\{u(0)\in\mathbb{R}^d:\, \dot{u}(t)-A_\lambda(t)u(t)=0,\,u(t)\rightarrow 0, t\rightarrow\infty\}\\
E^u_\lambda(0)&=\{u(0)\in\mathbb{R}^d:\, \dot{u}(t)-A_\lambda(t)u(t)=0,\,u(t)\rightarrow 0, t\rightarrow -\infty\}
\end{split}
\end{align}
denote the stable and unstable spaces of \eqref{lin-Hamilton}. Another natural way to think about solutions of \eqref{lin-Hamilton} is to consider the kernel of the bounded linear operator
\begin{align}\label{Lintroduction}
L_\lambda:H^1(\R,\R^d)\rightarrow L^2(\R,\R^d),\quad L_\lambda u=\dot{u}-A_\lambda(\cdot)u.
\end{align}
The operators $L_\lambda$ are Fredholm of index $0$ under common assumptions, which allows to study bifurcation for \eqref{Hamilton} by the parity as introduced by Fitzpatrick and Pejsachowicz in a series of papers in the 1980s and 90s (see \cite{Memoirs} for a comprehensive introduction). The parity $\sigma(L,[a,b])$ is a $\mathbb{Z}_2$-valued additive homotopy invariant for paths $L=\{L_\lambda\}_{\lambda\in[a,b]}$ of bounded Fredholm operators of index $0$ between two Banach spaces $X$ and $Y$, which came up in the construction of an extension of the classical Leray-Schauder degree to a homotopy invariant degree for quasilinear Fredholm mappings. Roughly speaking, the parity counts the dimensions of the kernels of the operators $L_\lambda$ modulo $2$ when the parameter $\lambda$ traverses the interval $[a,b]$. In particular, a non-trivial parity shows that $L_{\lambda^\ast}$ has a non-trivial kernel for some parameter $\lambda^\ast$. Its most important application in bifurcation theory is as follows (cf. \cite{FiPejsachowiczProc}).
\begin{theorem*}
Let $X,Y$ be Banach spaces and $G:[a,b]\times X\rightarrow Y$ a $C^1$ map such that $G(\lambda,0)=0$ for all $\lambda\in[a,b]$. Assume that the derivatives $L_\lambda:=D_uG(\lambda,0):X\rightarrow Y$ are Fredholm of index $0$ for all $\lambda\in[a,b]$, as well as invertible for $\lambda\in \{a,b\}$. If the parity $\sigma(L,[a,b])$ of $L=\{L_\lambda\}_{\lambda\in[a,b]}$ is non-trivial in $\mathbb{Z}_2$, then there is a bifurcation from the trivial branch $u\equiv 0$ for the equations
\[G(\lambda,u)=0,\]
i.e., there is some $\lambda^\ast\in(a,b)$ such that in every neighbourhood of $(\lambda^\ast,0)$ in $[a,b]\times X$, there is some $(\lambda,u)$ such that $G(\lambda,u)=0$ and $u\neq 0$.
\end{theorem*}
Thus, to find criteria for the existence of bifurcation for \eqref{Hamilton}, we are left with the problem to compute the parity $\sigma(L,[a,b])$ of the path of operators $L$ in \eqref{Lintroduction}. Pejsachowicz considered this question in \cite{JacoboAMS} by using a link between the parity and the (Atiyah-J\"anich-)index bundle for closed paths of bounded Fredholm operators of index $0$ that was pointed out by Fitzpatrick and Pejsachowicz in \cite{FiPejsachowiczII}.\\
In general, the index bundle is an element of the reduced $\KO$-theory group $\widetilde{\KO}(\Lambda)$, which stems from algebraic topology and is made by vector bundles over the compact topological space $\Lambda$. Atiyah and J\"anich showed (cf. \cite{KTheoryAtiyah}, \cite{Jaenich}) that any family of Fredholm operators parameterised by $\Lambda$ yields an element of $\widetilde{\KO}(\Lambda)$ which formally has several properties of the classical integral Fredholm index when addition in $\mathbb{Z}$ is replaced by the group operation in $\widetilde{\KO}(\Lambda)$. In particular, it is invariant under homotopies of the family of Fredholm operators, and on Hilbert spaces the homotopy class of the operator family is uniquely determined by this element in $\widetilde{\KO}(\Lambda)$.\\
For $\Lambda=S^1$ a family actually is a closed path and moreover $\widetilde{\KO}(S^1)\cong\mathbb{Z}_2$. Thus the parity and the index bundle are two ways to assign an element in $\mathbb{Z}_2$ to closed paths of Fredholm operators of index $0$, and Fitzpatrick and Pejsachowicz showed in \cite{FiPejsachowiczII} that these numbers coincide. In summary, if we consider \eqref{Hamilton} under the additional assumption that $A_a(t)=A_b(t)$ for all $t\in\mathbb{R}$ in \eqref{lin-Hamilton}, then the index bundle of Atiyah and J\"anich can be used as a bifurcation invariant.\\ To make sure that the operators $L_\lambda$ in \eqref{Lintroduction} are Fredholm of index $0$, Pejsachowicz assumed in \cite{JacoboAMS} that the limits
\begin{align}\label{limitshyp}
A_\lambda(\pm\infty)=\lim_{t\rightarrow\pm\infty} A_\lambda(t)
\end{align}
exist uniformly in $\lambda\in S^1$ and are hyperbolic, i.e., their spectra are disjoint to the imaginary axis. Moreover, there is supposed to be some $\lambda_0\in S^1$ such that there is a direct sum decomposition $E^u_{\lambda_0}(0)\oplus E^s_{\lambda_0}(0)=\mathbb{R}^d$, which implies that $L_{\lambda_0}$ is invertible. The main theorem of \cite{JacoboAMS} shows that the index bundle of $L=\{L_\lambda\}_{\lambda\in S^1}$ is non-trivial, and hence there is a bifurcation of \eqref{Hamilton}, when two vector bundles $E^s(+\infty)$ and $E^s(-\infty)$ over $S^1$ are not isomorphic. The fibres of these bundles are given by
\begin{align}\label{stableinf}
E^{s}_\lambda(\pm\infty)=\{u(0)\in\mathbb{R}^d: \dot{u}(t)-A_\lambda(\pm\infty)u(t)=0,\,\lim_{t\rightarrow\infty}u(t)=0\},
\end{align}
which are the generalised eigenspaces of $A_\lambda(\pm\infty)$ with respect to eigenvalues having negative real parts.\\
This remarkable link between classical analysis and algebraic topology has one major shortcoming: Computing the parity by the index bundle is only possible as $S^1$ is non-trivial as a topological space which is reflected by the fact that $\widetilde{\KO}(S^1)\cong\mathbb{Z}_2$ is non-trivial. For non-closed paths, $\widetilde{\KO}([a,b])$ is trivial and thus the index bundle cannot yield any information about the parity. 
A first attempt to improve Pejsachowicz' approach topologically was made by Hu and Portaluri in \cite{HuP}, where it is no longer required that $A_a(t)=A_b(t)$ for all $t\in\mathbb{R}$ in \eqref{lin-Hamilton}, but instead the authors assume in addition to the hyperbolicity of the limits in \eqref{limitshyp} that
\begin{align}\label{transversal}
E^u_{\lambda}(0)\oplus E^s_{\lambda}(0)=\mathbb{R}^d
\end{align}
at the endpoints $\lambda=a,b$ of the parametrising interval $[a,b]$. There is a well known gluing construction for vector bundles that yields bundles over $S^1$ from bundles over compact intervals \cite{Hatcher2}. Hu and Portaluri used this to obtain from the spaces \eqref{stableinf} for $\lambda\in[a,b]$ two bundles $E^s(+\infty)$, $E^s(-\infty)$ over $S^1$. Their main theorem shows that the parity of the path of operators $L_\lambda$ in \eqref{Lintroduction} is non-trivial if these bundles over $S^1$ are non-isomorphic. If in addition $A_a(t)=A_b(t)$ for all $t\in\mathbb{R}$, i.e., the path of Fredholm operators $L=\{L_\lambda\}_{\lambda\in[a,b]}$ is closed, there is no effect by the gluing construction which yields Pejsachowicz' theorem \cite{JacoboAMS}.\\
The aim of this paper is to substantially generalise the previous works on topological bifurcation theory for \eqref{Hamilton} concerning both the analytical and topological framework, which also motivates new research in multiparameter bifurcation theory by $K$-theoretical methods.
Firstly, using recent joint work of the first author with Longo and P\"otzsche, we consider \eqref{Hamilton} without even assuming the limits in \eqref{limitshyp} to exist. Instead we require the existence of exponential dichotomies, which is a concept that was already introduced by Perron in $1930$. As exponential dichotomies are a very natural and at the same time quite general setting for studying homoclinic trajectories of differential equations, it was already suggested by Pejsachowicz in \cite{JacoboAMS} to prove his theorem under this assumption. Let us point out that any system \eqref{lin-Hamilton} for which the limits \eqref{limitshyp} exist and are hyperbolic has the exponential dichotomies that we need below. Moreover, by definition an exponential dichotomy yields a splitting of the extended state space for linear non-autonomous differential equations into two bundles made by the stable and unstable spaces \eqref{stableunstable}. This naturally suggests to use methods from $K$-theory to study the equations \eqref{lin-Hamilton} in this setting. Secondly, the main achievement of Hu and Portaluri's work \cite{HuP} is to extend Pejsachowicz' theorem to systems \eqref{lin-Hamilton}, where the path of operators $L=\{L_\lambda\}$ in \eqref{Lintroduction} is not necessarily closed but has invertible endpoints. Here we go much beyond the state of the art and consider \eqref{lin-Hamilton} parametrised by any compact topological space $\Lambda$, where we assume in addition that $\Lambda_0\subset\Lambda$ is a closed subspace such that \eqref{transversal} holds for all $\lambda\in\Lambda_0$. This contains Pejsachowicz' setting for $\Lambda=S^1$, $\Lambda_0=\{\lambda_0\}$ and Hu and Portaluri's setting for $\Lambda=[a,b]$, $\Lambda_0=\{a,b\}$. We consider the index bundle in the relative $K$-theory group $\KO(\Lambda,\Lambda_0)$ as in the third author's previous works \cite{Wat15, NilsGeod, NilsLag}, and obtain as main achievement a surprisingly plain index theorem in terms of vector bundles made by the spaces \eqref{stableunstable} over the parameter space $\Lambda$. If $(\Lambda,\Lambda_0)=(S^1,\{\lambda_0\})$ or $(\Lambda,\Lambda_0)=([a,b],\{a,b\})$, then $\KO(\Lambda,\Lambda_0)\cong\mathbb{Z}_2$ and the previous results from \cite{JacoboAMS} and \cite{HuP} are immediate consequences of our theorem (see Section \ref{section-singlepar} below). Let us emphasize that Hu and Portaluri also consider the relative group $\KO([a,b],\{a,b\})$ in the beginning of their proof, but eventually their argument is based on the linear order of the interval $[a,b]$ and thus does not transfer to our setting.\\
Fitzpatrick and Pejsachowicz studied multiparameter bifurcation problems by the index bundle in $K$-theory in various papers, e.g., \cite{FiPejsachowiczII, Jacobo, JacoboTMNAI}. Applications still hinge on the question whether $\widetilde{\KO}(\Lambda)$ is non-trivial for the compact parameter space $\Lambda$, which is a rather unnatural assumption (see however \cite{Jacobo} and \cite{NilsBif}, where families of differential equations on bounded domains are parameterised by Grassmannians).
 A further aim of this paper is to initiate a novel approach to multiparameter bifurcation theory by relative $K$-theory $\KO(\Lambda,\Lambda_0)$ that will be continued in the upcoming work \cite{MultBif}. As an outcome, contractible spaces like higher dimensional intervals $\Lambda=I^k$ are now admissible parameter spaces relative to their boundary $\Lambda_0=\partial I^k$, and multiparameter bifurcation for \eqref{Hamilton} can be found by our main theorem of the present work, which will also play an important role in \cite{MultBif}.\\
Our paper is structured as follows. In the following second section we recap some preliminaries on dynamical systems and topological $K$-theory, where we in particular introduce the concept of exponential dichotomies as well as the index bundle in relative $K$-theory. Let us emphasize that we define vector bundles in terms of families of projections to make our exposition accessible to non-topologists. The third section is the core of our paper and is divided into four parts. We begin by a look at the differential operators $L_\lambda$ in \eqref{Lintroduction}, where we state our exact assumptions on \eqref{lin-Hamilton} in this paper, discuss the Fredholm property and compute the adjoints as unbounded operators in $L^2(\mathbb{R},\mathbb{R}^d)$ with dense domain $H^1(\R,\R^d)$. Secondly, we state our main theorem and some corollaries of it, which includes the case that the limits \eqref{limitshyp} exist and are hyperbolic. Thirdly, we consider the above mentioned cases where $\KO(\Lambda,\Lambda_0)\cong\mathbb{Z}_2$ and get the main theorems of Pejsachowicz \cite{JacoboAMS} and Hu-Portaluri \cite{HuP}. Finally, we prove our main theorem which actually is the longest part of our work and itself divided into four main steps. In the fourth section we discuss an example for a general compact space $\Lambda$, where we can compute our index in $\KO(\Lambda,\Lambda_0)$ explicitly in terms of eigenspaces of the matrices $A_\lambda$ in \eqref{lin-Hamilton}. If $\Lambda=S^1$ and $\Lambda_0=\{-1\}$ our example generalises Pejsachowicz' from \cite{JacoboAMS}. Hu and Portaluri's paper \cite{HuP} does not contain an example of their index theorem, which is a gap that we close by considering ours in the special case that $\Lambda=[0,\pi]$ and $\Lambda_0=\{0,\pi\}$. The final section of our paper deals with the bifurcation problem for \eqref{Hamilton}, where we consider the case of a compact contractible parameter space $\Lambda$ and where $\Lambda_0$ consists of two distinct points in $\Lambda$. Then $\KO(\Lambda,\Lambda_0)\cong\mathbb{Z}_2$ and a non-trivial index bundle yields a whole $1$-codimensional subset of the parameter space that consists of bifurcation points. Let us emphasize that the bifurcation problem is even more interesting in this setting if $\Lambda_0$ is connected but not a singleton, as the classical parity then fails as a bifurcation invariant. The main aim of the upcoming work \cite{MultBif} is to show that in contrast our relative $K$-theory approach still works and can even spot isolated bifurcation points in higher-dimensional parameter spaces.


\section{Preliminaries}
This first section is intended to review the analytical and topological concepts that are needed to understand our index theorem below.

\subsection{First Order Equations and Exponential Dichotomies}\label{section-expdich}
In this section, we aim to recap the concept of exponential dichotomy, which is a convenient substitute of hyperbolicity for non-autonomous dynamical systems.\\
Let $\mathbb{I}\subset \mathbb{R}$ be an unbounded interval and $A: \R\to \Mat(d,\mathbb{R})$ a continuous matrix-valued map, where $\Mat(d,\mathbb{R})$ denotes the space of all real $d\times d$ matrices. An invariant projector for the linear system
\begin{equation}\label{LLL}
\dot{x}=A(t)x
\end{equation}
is a family $P: \mathbb{I}\to \Mat(d,\mathbb{R})$ of projections (i.e., $P(t)^2=P(t)$ for all $t\in \mathbb{I}$) such that
\begin{equation}\label{invariance}
\Phi(t,s)P(s)=P(t)\Phi(t,s) \text{ for all }  t,s\in \mathbb{I},
\end{equation}
where $\Phi(t,s)\in \GL(d,\mathbb{R})$ denotes the transition matrix of \eqref{LLL}, i.e., the map $t\mapsto\Phi(t,s)$ satisfies \eqref{LLL} on all of $\mathbb{I}$ and $\Phi(s,s)=\id$ for all $s\in\mathbb{I}$. Let us note for later reference the common property
\begin{equation}\label{triplePhi}
\Phi(t,s)\Phi(s,r)=\Phi(t,r)\quad \text{ for all }t,s,r\in \mathbb{I},\, r\leq s\leq t.
\end{equation}
The following definition is fundamental for the main assumption on the equations that we consider below.
\begin{definition}\label{def-expdich}
The linear system \eqref{LLL} has an exponential dichotomy on $\mathbb{I}$ if there exists an invariant projector
$P$ and real numbers $K\geq 1$ and $\alpha>0$ such that
\begin{equation*}
\|\Phi(t,s)P(s)\|\leq Ke^{-\alpha(t-s)}\quad\text{and}\quad \|\Phi(s,t)(I_d-P(t))\|\leq Ke^{-\alpha(t-s)}\text{ for all }t\geq s\text{ with }t,s\in \mathbb{I}.
\end{equation*}
\end{definition}
If \eqref{LLL} has an exponential dichotomy on $\mathbb{I}=[\tau,+\infty)$, we set
\begin{equation*}
E^s(\tau)=\left\{x\in \R^d : \lim\limits_{t\to\infty}\Phi(t,\tau)x=0\right\},
\end{equation*}
and if there is an exponential dichotomy on $\mathbb{I}=(-\infty,\tau]$
\begin{equation*}
E^u(\tau)=\left\{x\in \R^d : \lim\limits_{t\to-\infty}\Phi(t,\tau)x=0\right\},
\end{equation*}
which is in accordance with our notation in \eqref{stableunstable}. It is important to note that these spaces and the invariant projector determine each other by
       \[
       R(P(t)) = E^s(t) \quad \text{for all } t \in \mathbb{I},
       \]
       when $\mathbb{I}=[t,\infty)$ and where $R(P(t))$ denotes the range of the projector $P(t)$, as well as
       \[
       N(P(t)) = E^u(t) \quad \text{for all } t \in \mathbb{I},
       \]
       when $\mathbb{I}=(-\infty,t]$, and where $N(P(t))$ denotes the kernel of the projector $P(t)$.
In particular, if there is an exponential dichotomy on all of $\mathbb{R}$, then this uniquely determines the projector $P$. Moreover, in this case $E^u(t)\cap E^s(t)=\{0\}$ for any $t\in\mathbb{R}$, which shows that \eqref{LLL} only has the trivial solution $x\equiv 0$ on the whole line.\\
As an example, note that an autonomous linear differential equation $\dot{x}=Ax$ has an exponential dichotomy on an unbounded interval $\mathbb{I}$ if and only if it is hyperbolic, i.e., the matrix $A\in\Mat(d,\mathbb{R})$ has no eigenvalues on the imaginary axis. Furthermore, the sets $E^s(t)$ and $E^u(t)$ are constant and given by the direct sum of the generalised eigenspaces corresponding to eigenvalues in
\[\sigma^-(A)=\{\mu\in\sigma(A):\,\RE(\mu)<0\}\quad\text{and}\quad \sigma^+(A)=\{\mu\in\sigma(A):\,\RE(\mu)>0\}.\]
Clearly, in this case the constants $K\geq 1$ and $\alpha>0$ in Definition \ref{def-expdich} can explicitly be determined by using the spectrum and the eigenvectors of the matrix $A$.\\
The following Theorem from \cite{coppel:78} is vital for showing that a given system has an exponential dichotomy.
\begin{theorem}\label{roughness}
Suppose that \eqref{LLL} has an exponential dichotomy with constants $K\geq 1$ and $\alpha>0$ on $\mathbb{I}$, where $\mathbb{I}=[\tau,\infty)$, $\mathbb{I}=(-\infty,\tau]$ for some $\tau\in\R$, or $\mathbb{I}=\mathbb{R}$. If $B: \R\to \Mat(d,\mathbb{R})$ is continuous and
\begin{equation}\label{estimation}
\sup_{t\in \mathbb{I}}|B(t)|<\frac{\alpha}{4K^2},
\end{equation}
then the perturbed system
\begin{equation*}
\dot{x}=(A(t)+B(t))x
\end{equation*}
has an exponential dichotomy on $\mathbb{I}$ as well.
\end{theorem}
Let us point out that stronger perturbation results for exponential dichotomies can be found, e.g., in \cite[Theorem~3.1 and Corollary~3.1]{JuWiggins2001}. Here we stick to Coppel's classical Theorem \ref{roughness} for the sake of a simple presentation of our results.\\
It was also shown by Coppel in \cite{coppel:78} that if $\dot{x} = A(t)x$ has an exponential dichotomy on an interval  $[\tau_1, \infty)$, then it also has an exponential dichotomy on any larger interval $[\tau_2, \infty)$ with $\tau_2 \leq \tau_1$. The same holds for exponential dichotomies on intervals of the form $(-\infty, \tau]$.\\
Let us recall that a matrix-valued map $A:\mathbb{R}\rightarrow\Mat(d,\mathbb{R})$ is called asymptotically hyperbolic if the limits $A_{\pm}:= \lim_{t \to \pm \infty} A(t)$ exist and are hyperbolic (cf. \cite{AlbertoODE}). Then \eqref{LLL} has an exponential dichotomy on both $[\tau,\infty)$ and $(-\infty,-\tau]$ for any $\tau\geq 0$, which directly follows from the above as

\begin{equation*}
\dot{x} = A(t)x=(A_\pm+B(t))x \text{ for }B(t):=A(t)-A_\pm.
\end{equation*}
The next result links the existence of an exponential dichotomy on all of $\R$ to the exponential dichotomies on both half lines.
\begin{theorem}\label{thm-invertible}
The linear system $\eqref{LLL}$ possesses an exponential dichotomy on $\mathbb{I}=\R$ if and only if
\begin{enumerate}[$(a)$]
\item $\eqref{LLL}$ admits exponential dichotomies on both $\R^+_0:=[0,\infty)$ and $\R^-_0:=(-\infty,0]$ with respective projectors $P^+$ and $P^-$,
\item $R(P^+(0))\oplus N(P^-(0))=\R^d$.
\end{enumerate}
\end{theorem}
Finally, with the problem \eqref{LLL}, we associate the dual differential equation
\begin{equation}\label{A-P-L}
\dot{x}=-A^{\top}(t)x,
\end{equation}
 where $A(t)^{\top}$ denotes the transpose of the matrix $A(t)$. It can easily be shown (see \cite{Poetzsche-24}) that if \eqref{LLL} has an exponential dichotomy on $\mathbb{I}$ with constants $K\geq 1$, $\alpha>0$, and projector $P: \mathbb{I}\to \Mat(d,\mathbb{R})$, then \eqref{A-P-L} also has an exponential dichotomy on $\mathbb{I}$ with the same constants $K\geq 1$, $\alpha>0$, and the projector $P^{\ast}: \mathbb{I}\to \Mat(d,\mathbb{R})$ is given by
 \begin{equation}\label{dual-operator}
 P^{\ast}(t)=I_d-P(t)^{\top}.
 \end{equation}
 Consequently, for any $\tau\geq 0$ we see that the stable and unstable spaces of \eqref{A-P-L} are
 \begin{align*}
\left\{x(\tau)\in \mathbb{R}^d: \dot{x}(t)=-A^{\top}(t)x(t),\;x(t)\xrightarrow[t\rightarrow+\infty]{}0\right\}&=E^s(\tau)^{\perp} \text{ for }
\mathbb{I}=[\tau,\infty),\\
\left\{x(-\tau)\in \mathbb{R}^d: \dot{x}(t)=-A^{\top}(t)x(t),\;x(t)\xrightarrow[t\rightarrow-\infty]{}0\right\}&=E^u(-\tau)^{\perp} \text{ for }
\mathbb{I}=(-\infty,-\tau].
\end{align*}
Henceforth we consider parameterised linear differential equations of the form
\begin{equation}\label{P-L}
\dot{x}=A_{\lambda}(t)x,
\end{equation}
where $\Lambda$ is a compact connected topological space, $A:\Lambda\times\mathbb{R}\to\Mat(d,\mathbb{R})$ is continuous and each map $A_\lambda:=A(\lambda,\cdot):\mathbb{R}\rightarrow\Mat(d,\mathbb{R})$ is bounded and uniformly continuous. We will assume  that for any $\lambda\in\Lambda$, the problem \eqref{P-L} admits exponential dichotomies on both intervals $\R^+_0$ and  $\R^-_0$ with respective projectors $P^+_{\lambda}$ and $P^-_{\lambda}$.
The following theorem was recently proved by the first author in collaboration with P\"otzsche and Longo (cf.~\cite{Poetzsche-24, Longo}).
\begin{theorem}\label{continuous-family-projections}
Under the above assumptions, the invariant projectors $P^-_\lambda(-\tau)$ and $P^+_\lambda(\tau)$ can be chosen such that the mappings
$\lambda\longmapsto P^-_{\lambda}(-\tau)\in\Mat(d,\mathbb{R})$ and $\lambda\longmapsto P^+_{\lambda}(\tau)\in\Mat(d,\mathbb{R})$  are continuous on $\Lambda$ for any $\tau\geq 0$.
\end{theorem}
Thus the stable and unstable subspaces
\begin{align}\label{continuous-family-projectionsII}
\begin{split}
R(P_{\lambda}^+(\tau))&=E^s_{\lambda}(\tau)\\
R(I_d-P^-_{\lambda}(-\tau))&=N(P_{\lambda}^-(-\tau))=E^u_{\lambda}(-\tau)
\end{split}
\end{align}
of \eqref{P-L} depend in a way continuously on the parameter $\lambda\in\Lambda$, which opens up the above setting for topological methods by the well-known fact that the images of families of projections yield vector bundles (see, e.g.,~\cite[Prop. 1.7.5]{Park}). To be more precise, we consider vector bundles over a compact topological space $\Lambda$ as subsets of the form
\[E=\{(\lambda,u)\in\Lambda\times\mathbb{R}^d:\, u\in R(P_\lambda)\}\subset\Lambda\times\mathbb{R}^d\]
of the topological space $\Lambda\times\mathbb{R}^d$ for some $d\in\mathbb{N}_0:=\mathbb{N}\cup\{0\}$, where $P:\Lambda\rightarrow\Mat(d,\mathbb{R})$ is a continuous family of idempotent matrices, i.e., projections on $\mathbb{R}^d$. We denote by $E_\lambda=R(P_\lambda)$ the fibre over a point $\lambda\in\Lambda$, which is a linear subspace of $\mathbb{R}^d$. If $\Lambda$ is connected, the dimension $\dim(E_\lambda)$ does not depend on $\lambda$ and is called the dimension of the bundle $E$.\\
When $P$ is a constant family and $R(P_\lambda)=V\subset\R^d$, $\lambda\in \Lambda$, we shorten notation by denoting the product $\Lambda\times V$ by $\Theta(V)$. Note that in particular $\Theta(\R^d)$ is the bundle obtained from the family $P_\lambda=I_d$, $\lambda\in\Lambda$. We say that a vector bundle $F$ with family of projections $\tilde{P}:\Lambda\rightarrow\Mat(d,\mathbb{R})$ is a subbundle of $E$ if $R(\tilde{P}_\lambda)\subset R(P_\lambda)$ for all $\lambda\in\Lambda$, i.e., each fibre $F_\lambda$ is a linear subspace of $E_\lambda$.\\
If $E,F$ are two vector bundles over $\Lambda$, then a continuous map $h:E\rightarrow F$ is called a bundle homomorphism if it preserves fibres (i.e., $h(E_\lambda)\subset F_\lambda$, $\lambda\in\Lambda$) and $h|_{E_\lambda}:E_\lambda\rightarrow F_\lambda$ is linear for all $\lambda\in\Lambda$. As usual, a bijective homomorphism is called an isomorphism, and clearly this is the case if and only if every $h|_{E_\lambda}:E_\lambda\rightarrow F_\lambda$ is an isomorphism.\\
Now, by Theorem \ref{continuous-family-projections} and \eqref{continuous-family-projectionsII}, for any $\tau\geq 0$, the sets
\begin{align*}
E^s(\tau):=\{(\lambda,v)\in \Lambda\times \R^d : v\in E^s_{\lambda}(\tau)\} \text{ and }
E^u(-\tau):=\{(\lambda,v)\in \Lambda\times \R^d : v\in E^u_{\lambda}(-\tau)\}
\end{align*}
are vector bundles over $\Lambda$. Let us point out for later reference that this in particular implies that $\dim(E^u_\lambda(0))$ and $\dim(E^s_\lambda(0))$ are constant, and
\begin{align}\label{dimform}
\dim(E^u_\lambda(0))+\dim(E^s_\lambda(0))=d
\end{align}
by Theorem \ref{thm-invertible} if there is some $\lambda\in\Lambda$ for which \eqref{P-L} has an exponential dichotomy on all of $\mathbb{R}$.\\
Finally, it follows from well-known results on the continuity of solutions with respect to parameters (cf. \cite{Amann}) that the map $\mathbb{X} \colon \Lambda \times \R^2 \to \GL(d,\mathbb{R})$ defined by $\mathbb{X}(\lambda, (t,s)) := \Phi_\lambda(t,s)$ is continuous, where $\Phi_\lambda$ denotes the transition matrix of \eqref{P-L}. As
\[
E^s_{\lambda}(\tau) = \mathbb{X}_{\lambda}(\tau,0) E^s_{\lambda}(0) \quad \text{and} \quad E^u_{\lambda}(-\tau) = \mathbb{X}_{\lambda}(-\tau,0) E^u_{\lambda}(0),\quad\lambda\in\Lambda,
\]
we obtain bundle isomorphisms
\begin{equation}\label{ISO-S-U}
E^s(\tau) \cong E^s(0) \quad \text{and} \quad E^u(-\tau) \cong E^u(0) \quad \text{for all } \tau \geq 0,
\end{equation}
that will be of importance below.


\subsection{The Index Bundle in Relative $K$-Theory}\label{section-vb}
The aim of this section is to provide an introduction to relative $K$-theory groups for real vector bundles (called $\KO$-theory) and a variant of the well-known Atiyah-J\"anich index bundle in this setting. We restrict to the main concepts and properties that are necessary to understand our main Theorem \ref{thm-main-lin} below, and postpone a comprehensive introduction to our parallel work \cite{MultBif} that deals with relative $\KO$-theory and multiparameter bifurcation.\\
We have already briefly recalled the concept of a vector bundle in the previous section, where it may have looked tailored to our purposes. Actually, up to isomorphism, the given definition is equivalent to the more common one along local triviality (cf. \cite{KTheoryAtiyah}, \cite{Park}). Thus, as we only need isomorphism classes of vector bundles below, we do not lose any generality if we stick to the previous definition and think of (real) vector bundles over a compact topological space $\Lambda$ as subsets of a product $\Lambda\times\mathbb{R}^d$ of the form
\[E=\{(\lambda,u)\in\Lambda\times\mathbb{R}^d:\, u\in R(P_\lambda)\},\]
where $P:\Lambda\rightarrow\Mat(d,\mathbb{R})$ is a continuous family of projections on $\mathbb{R}^d$ and $d\in\mathbb{N}_0$. Henceforth we need the following sum operation on the set of all vector bundles over $\Lambda$. If $E$, $F$ are defined by families of projections $P_E:\Lambda\rightarrow\Mat(p_1,\mathbb{R})$ and $P_F:\Lambda\rightarrow\Mat(p_2,\mathbb{R})$, then
\[P_{E\oplus F}=\begin{pmatrix}
P_E&0\\
0&P_F
\end{pmatrix}\]
is a family of projections on $\mathbb{R}^{p_1+p_2}$. The bundle $E\oplus F$ induced by this family is called the direct sum bundle of $E$ and $F$, and its fibres $(E\oplus F)_\lambda$ are the direct sums of $E_\lambda\subset\mathbb{R}^{p_1}$ and $F_\lambda\subset\mathbb{R}^{p_2}$ considered as subspaces of $\mathbb{R}^{p_1+p_2}$.


\subsubsection{Relative $K$-Theory}
In this section we briefly recall the construction of the relative $\KO$-theory groups and its main properties, where we mainly follow \cite[\S 10]{Hus}. Henceforth we assume that $\Lambda$ is a compact topological space and $\Lambda_0\subset \Lambda$ a closed subset. We consider triples $\xi=\{E_0,E_1,h\}$, where $E_0$ and $E_1$ are vector bundles over $\Lambda$ and $h: E_0\to E_1$ is a bundle homomorphism such that the restriction $h|_{\Lambda_0}:E_0|_{\Lambda_0}\rightarrow E_1|_{\Lambda_0}$ to $\Lambda_0$ is an isomorphism. If $h: E_0\to E_1$ is an isomorphism (on all of $\Lambda$), the element $\xi=\{E_0,E_1,h\}$ is called trivial.\\
Two triples $\xi^1=\{E_0^1,E_1^1,h_1\}$ and $\xi^2=\{E_0^2,E_1^2,h_2\}$ are isomorphic ($\xi^1\cong\xi^2$) if there exist bundle isomorphisms $\varphi_0: E_0^1\to E_0^2$ and $\varphi_1: E_1^1\to E_1^2$ such that the diagram
\begin{align}\label{ISO1}
\begin{split}
\xymatrix{
E_0^2\ar[r]^{h_2}& E_1^2\\
E_0^1\ar[r]^-{h_1} \ar[u]^{\varphi_0}& \ar[u]_{\varphi_1} E_1^1
}
\end{split}
\end{align}
commutes over $\Lambda_0$. It is readily seen that this is an equivalence relation, and henceforth we let $L(\Lambda,\Lambda_0)$ be the set of isomorphism classes. The elements of $L(\Lambda,\Lambda_0)$ will be denoted by round brackets, i.e., $(E_0,E_1,h)$. Note that $L(\Lambda,\Lambda_0)$ is a commutative semigroup under the operation $\oplus$ induced by the direct sum. Moreover, for any bundles $E_0$, $E_1$, the equivalence class $(E_0,E_1,h)$ only depends on the restriction of $h$ to $\Lambda_0$.\\
To turn $L(\Lambda,\Lambda_0)$ into a group, we define an equivalence relation by $\xi^1\sim \xi^2$ if there are trivial elements $\eta^1,\eta^2$ in $L(\Lambda,\Lambda_0)$ such that $\xi^1\oplus\eta^1\cong\xi^2\oplus\eta^2$. The set of equivalence classes is $\KO(\Lambda,\Lambda_0)$, the real relative $K$-theory of the compact pair $(\Lambda,\Lambda_0)$. Henceforth, we denote the equivalence class of $(E_0,E_1,h)\in L(\Lambda,\Lambda_0)$ in $\KO(\Lambda,\Lambda_0)$ by square brackets $[E_0,E_1,h]$, and our next aim is to briefly survey some elementary properties of $\KO(\Lambda,\Lambda_0)$ from \cite{Hus} that will in particular explain why $\KO(\Lambda,\Lambda_0)$ is an abelian group with respect to the operation
\begin{equation*}
[E_0^0,E_1^0,h_0]+[E_0^1,E_1^1,h_1]=[E_0^0\oplus E_0^1,E_1^0\oplus E_1^1,h_0\oplus h_1],
\end{equation*}
induced by the direct sum of vector bundles. Let us first note that a neutral element is given by any class $[E_0,E_1,h]$ where $h:E_0\rightarrow E_1$ is an isomorphism. Secondly, it is shown in \cite[\S 10, Cor. 5.4]{Hus} that
\begin{equation}\label{homotopy}
[E_0,E_1,h(0)]=[E_0,E_1,h(1)]\in \KO(\Lambda,\Lambda_0).
\end{equation}
for any two vector bundles $E_0$ and $E_1$ over $\Lambda$ and any continuous family of bundle morphisms
\begin{equation*}
h: [0,1]\to \hom(E_0,E_1)
\end{equation*}
such that $(h(t))_{\lambda}\in \GL((E_0)_{\lambda},(E_1)_{\lambda})$ for all $t\in[0,1]$ and $\lambda\in \Lambda_0$. It can be shown from \eqref{homotopy} that if $[E_0,E_1,h],[E_1,E_2,g]\in \KO(\Lambda,\Lambda_0)$, then
\begin{equation}\label{sum}
[E_0,E_1,h]+[E_1,E_2,g]=[E_0,E_2,g\circ h].
\end{equation}
If $[E_0,E_1,h]\in \KO(\Lambda,\Lambda_0)$, then by definition $h|_{\Lambda_0}:E_0|_{\Lambda_0}\rightarrow E_1|_{\Lambda_0}$ is an isomorphism and thus has an inverse $(h|_{\Lambda_0})^{-1}:E_1|_{\Lambda_0}\rightarrow E_0|_{\Lambda_0}$. As $\Lambda_0\subset \Lambda$ is closed, $(h|_{\Lambda_0})^{-1}$ can be extended to a bundle morphism $\widetilde{h}:E_1\rightarrow E_0$ (see \cite[\S 1.4]{KTheoryAtiyah}) and, as explained above, the class $[E_1,E_0,\widetilde{h}]$ does not depend on the extension from $\Lambda_0$ to $\Lambda$. Now it follows from \eqref{sum} that the inverse element of $[E_0,E_1,h]$ in $\KO(\Lambda,\Lambda_0)$ is given by
\begin{equation}\label{inverse-h}
-[E_0,E_1,h]:=[E_1,E_0,\widetilde{h}],
\end{equation}
and thus $\KO(\Lambda,\Lambda_0)$ indeed is a group.\\
Let us briefly recall (cf., e.g., \cite{Hatcher2}) that classically the reduced absolute $\KO$-theory group $\widetilde{\KO}(\Lambda)$ of a compact topological space $\Lambda$ is defined by an equivalence relation on the set of formal differences $[E]-[F]$, where $E,F$ are vector bundles of the same dimension over $\Lambda$ and the square bracket stands for the isomorphism class of a vector bundle. The equivalence relation is given by $[E_0]-[F_0]=[E_1]-[F_1]$ if and only if $E_0\oplus F_1\oplus\Theta(\mathbb{R}^k)$ is isomorphic to $E_1\oplus F_0\oplus\Theta(\mathbb{R}^k)$ for some $k\in\mathbb{N}_0$, and the group operation is defined by $([E_0]-[F_0])+([E_1]-[F_1])=[E_0\oplus E_1]-[F_0\oplus F_1]$. We note for later reference that for any fixed base point $\lambda_0\in\Lambda$ the map
\begin{align}\label{alpha}
\Psi_{\lambda_0}:\KO(\Lambda,\{\lambda_0\})\rightarrow\widetilde{\KO}(\Lambda),\quad [E,F,h]\mapsto [E]-[F]
\end{align}
is a group isomorphism by \cite[\S 10, Cor. 5.3]{Hus}.

When the parameter space $\Lambda$ is contractible (e.g., if $\Lambda$ is a convex subset of some Euclidean space), any vector bundle $E$ over $\Lambda$ is isomorphic to the product bundle $\Theta(\mathbb{R}^d)$ with total space $\Lambda\times \R^d$ for some $d\in\mathbb{N}_0$ (cf.~\cite{Hatcher2}). Hence if $[E_0,E_1,h]\in \KO(\Lambda,\Lambda_0)$ and $\Lambda_0\neq\emptyset$, then there is some $d\in\mathbb{N}_0$ such that
\begin{equation}\label{trivialisationbeta}
[E_0,E_1,h]=[\Theta(\R^d),\Theta(\R^d),\phi\circ h\circ \varphi^{-1}],
\end{equation}
where $\varphi: E_0\to \Theta(\R^d)$ and $\phi: E_1\to\Theta(\R^d)$ are bundle isomorphisms. Thus, if $\Lambda$ is contractible, we can assume that elements in $\KO(\Lambda, \Lambda_0)$ are of the form $[\Theta(\R^d), \Theta(\R^d), h]$, where $h=\{h_\lambda\}_{\lambda\in\Lambda}$ is a family of linear maps $h_\lambda: \mathbb{R}^d \to \mathbb{R}^d$ such that $h_\lambda\in\GL(d,\mathbb{R})$ for $\lambda\in\Lambda_0$. The following theorem will be important in applications in Section \ref{section-singlepar}.

\begin{proposition}\label{beta-iso}
Let $\Lambda$ be a contractible compact topological space and $\Lambda_0=\{\lambda_0,\lambda_1\}\subset\Lambda$ for two distinct elements $\lambda_0, \lambda_1$ of $\Lambda$. Then the map $\psi_{\lambda_0,\lambda_1}: \KO(\Lambda,\Lambda_0)\to \Z_2=\{0,1\}$ defined by
\begin{equation}\label{beta}
(-1)^{\psi_{\lambda_0,\lambda_1}([\Theta(\R^d),\Theta(\R^d),h])}=\sgn \det h_{\lambda_0}\cdot \sgn\det h_{\lambda_1}
\end{equation}
is an isomorphism.
\end{proposition}
\begin{proof}
We leave it to the reader to check that $\psi_{\lambda_0,\lambda_1}$ is a well-defined group homomorphism, which follows by straightforward computations from basic properties of the determinant. To see that $\psi_{\lambda_0,\lambda_1}$ is bijective, let us firstly recall that $[\Theta(\R^d),\Theta(\R^d),h]$ only depends on $h$ by the restriction of $h$ to $\Lambda_0$, and thus in our case it is uniquely determined by the two invertible matrices $h_{\lambda_0}$ and $h_{\lambda_1}$. Now any continuous map $\hat{h}:I\times\Lambda_0\rightarrow\GL(d,\mathbb{R})$ such that $\hat{h}(0,\lambda_0)=h_{\lambda_0}$ and $\hat{h}(0,\lambda_1)=h_{\lambda_1}$ extends to a continuous map $\hat{h}:I\times\Lambda\rightarrow\Mat(d,\mathbb{R})$. The homotopy invariance \eqref{homotopy} and the fact that $\GL(d,\mathbb{R})$ has two path components now yield that all elements of $\KO(\Lambda,\Lambda_0)$ are of the form $[\Theta(\R^d),\Theta(\R^d),h^i]$, $i=1,2$, for either $h^1(\lambda_0)=I_d=h^1(\lambda_1)$ or $h^2(\lambda_0)=I_d$, $h^2(\lambda_1)=- I_d$. Note that these are indeed all cases as in general $(E,F,h)=(E,F,-h)$ for any element in $L(\Lambda,\Lambda_0)$. This shows the claim as $\psi_{\lambda_0,\lambda_1}([\Theta(\R^d),\Theta(\R^d),h^1])=0$ and $\psi_{\lambda_0,\lambda_1}([\Theta(\R^d),\Theta(\R^d),h^2])=1$.
\end{proof}

\subsubsection{The Index Bundle}\label{subsec-indexbundle}
Now we shall recall the construction of the index bundle for maps
\[
L: (\Lambda,\Lambda_0)\rightarrow (\Phi_0(X,Y),\GL(X,Y)),
\]
where as before $\Lambda$ is compact, $\Lambda_0\subset\Lambda$ is closed and $\Phi_0(X,Y)$ denotes the set of Fredholm operators of index $0$ between two Banach spaces $X$ and $Y$. Our main references for this section are \cite{indbundleIch}, \cite[\S 2.1]{Wat15} and the upcoming work \cite{MultBif}.\\
By \cite[Lemma 2.1]{Wat15}, there is a finite dimensional subspace $V\subset Y$ such that
\begin{align}\label{subspace}
R(L_\lambda)+V=Y,\quad\lambda\in\Lambda,
\end{align}
i.e., $V$ is transversal to the images of all operators $L_{\lambda}$. As $V$ is of finite dimension, there is a bounded projection $P$ onto $V$, which makes the composition
\begin{align*}
X\xrightarrow{L_\lambda}Y\xrightarrow{I_Y-P} R(I_Y-P)
\end{align*}
surjective and
\begin{align}\label{ker-I-P}
N((I_Y-P)\circ L_\lambda)=L_{\lambda}^{-1}(V)\text{ as well as } \dim N((I_Y-P)\circ L_\lambda)=\dim V,
\end{align}
for all $\lambda\in\Lambda$. Consequently, by \cite[Prop.~14.2.3]{Dieck},
\begin{equation}\label{E-V-L}
E(L,V):=\{(\lambda,w)\in \Lambda\times X\mid w\in L_{\lambda}^{-1}(V)\}
\end{equation}
is a vector bundle of dimension $\dim(E(L,V))=\dim(V)$ over $\Lambda$. The map $L$ restricts to a bundle morphism
$L|_{E(L,V)}: E(L,V)\to \Theta(V)$, where $\Theta(V)$ stands for the product bundle $\Lambda\times V$ over $\Lambda$, and thus yields a $\KO$-theory class
\begin{align}\label{defiindbund}
\ind(L):=[E(L,V),\Theta(V),L|_{E(L,V)}]\in \KO(\Lambda,\Lambda_0).
\end{align}
Note that $(L|_{E(L,V)})_\lambda$ is an isomorphism if and only if $L_\lambda$ is an isomorphism, which shows that $\ind(L)$ indeed is in $\KO(\Lambda,\Lambda_0)$. The class \eqref{defiindbund} does not depend on the choice of the finite dimensional space $V$ in \eqref{subspace} (cf. \cite[Thm. 3]{indbundleIch}), and if $\Lambda_0=\{\lambda_0\}$ for some $\lambda_0\in\Lambda$, then
\begin{align}\label{IndBundClassical}
\Psi_{\lambda_0}(\ind(L))=[E(L,V)]-[\Theta(V)]\in\widetilde{\KO}(\Lambda)
\end{align}
is the classical Atiyah-J\"anich bundle (cf. \cite{KTheoryAtiyah}, \cite{Jaenich}), where $\Psi_{\lambda_0}$ is the canonical isomorphism in \eqref{alpha}. The following properties of the index bundle can be found in \cite{indbundleIch}, and most of them directly follow from its definition.
\begin{proposition}\label{index-bundle}
The element $\ind(L)\in \KO(\Lambda,\Lambda_0)$ has the following properties:
\begin{enumerate}[$(i)$]
\item  If $L_\lambda$ is invertible for all $\lambda\in\Lambda$, then $\ind(L)=0\in \KO(\Lambda,\Lambda_0)$.
\item  If $H:[0,1]\times(\Lambda,\Lambda_0)\rightarrow (\Phi_0(X,Y),\GL(X,Y))$ is a homotopy of Fredholm operators, then $$\ind(H(0,\cdot))=\ind(H(1,\cdot))\in \KO(\Lambda,\Lambda_0).$$
\item  If $S:(\Lambda,\Lambda_0)\to (\Phi_0(Y,Z),\GL(Y,Z))$ and $L: (\Lambda,\Lambda_0)\to (\Phi_0(X,Y),$ $\GL(X,Y))$ are two families of Fredholm operators, then
\begin{equation}\label{logaritmmic}
\ind(S\diamond L)=\ind(S)+\ind(L)\in \KO(\Lambda,\Lambda_0),
\end{equation}
where $(S\diamond L)_{\lambda}=S_{\lambda}\circ L_{\lambda}$ for $\lambda\in\Lambda$.
\item If $L_1: (\Lambda,\Lambda_0)\to (\Phi_0(X_1,Y_1),\GL(X_1,Y_1))$ and $L_2: (\Lambda,\Lambda_0)\to (\Phi_0(X_2,Y_2),\GL(X_2,Y_2))$ are two families of Fredholm operators, then
\begin{equation}\label{additivity}
\ind(L_1\oplus L_2)=\ind(L_1)+\ind(L_2)\in \KO(\Lambda,\Lambda_0).
\end{equation}
\end{enumerate}
\end{proposition}
In the proof of Theorem \ref{thm-main-lin} we need the following mild generalisation of the definition of the index bundle. If $P:\Lambda\rightarrow\mathcal{L}(X)$ is a continuous map of projections on $X$, i.e., $P^2_\lambda=P_\lambda$, $\lambda\in \Lambda$, then we call the subset
\begin{equation}\label{Banach-P}
X_P=\{(\lambda,u)\in\Lambda\times X:\, P_\lambda u=u\}
\end{equation}
a Banach subbundle of the trivial bundle $\Theta(X)$. For a family $L:\Lambda\to\mathcal L(X,Y)$ we define a bundle morphism
\[
L_P:X_P\to\Theta(Y),\qquad L_P(\lambda,u):=(\lambda,L_\lambda u).
\]
We say that $L_P$ is a Fredholm morphism of index $0$ if
$
L_\lambda\big|_{(X_P)_\lambda}\in \Phi_0\big((X_P)_\lambda,Y\big) \text{ for all }\lambda\in\Lambda,
$
where $(X_P)_\lambda=\{u\in X:\ P_\lambda u=u\}=\mathrm{Im}(P_\lambda)$ is the image of the projection $P_\lambda$. The construction of \eqref{defiindbund} as well as all properties in Proposition \ref{index-bundle} still hold in this slightly more general setting. A thorough discussion of the index bundle for morphisms between general Banach bundles can be found in \cite{indbundleIch}.


\section{The Family Index Theorem}

\subsection{The Differential Operators and their Adjoints}
In this section we focus on the linear systems
\begin{equation}\label{Lin}
\left\{
\begin{array}{l}
\dot{x}= A_\lambda(t)x \\
\lim\limits_{t\to\pm\infty} x(t)=0
\end{array}
\right.
\end{equation}
where our standing assumptions are that $\lambda$ is a parameter in a compact and connected topological space $\Lambda$ and $A:\Lambda\times\mathbb{R}\rightarrow\Mat(d,\mathbb{R})$ is a continuous family of real quadratic matrices of dimension $d\in\mathbb{N}$ such that each $A_\lambda:\mathbb{R}\rightarrow\Mat(d,\mathbb{R})$ is bounded and uniformly continuous. We also formally consider the adjoint equations
\begin{equation}\label{ALin}
\left\{
\begin{array}{l}
\dot{x}=-A^\top_\lambda(t)x \\
\lim\limits_{t\to\pm\infty} x(t)=0.
\end{array}
\right.
\end{equation}
Our main objects of study are the differential operators
\begin{align}\label{L}
L_\lambda:H^1(\mathbb{R},\mathbb{R}^d)\rightarrow L^2(\mathbb{R},\mathbb{R}^d),\quad (L_\lambda x)(t)=\dot{x}(t)-A_\lambda(t)x(t) \text{ for }\lambda\in\Lambda
\end{align}
as well as the adjoint operators given by
\begin{align}\label{DL}
L_\lambda^{\ast}:H^1(\mathbb{R},\mathbb{R}^d)\rightarrow L^2(\mathbb{R},\mathbb{R}^d),\quad (L_\lambda^{\ast} x)(t)=-\dot{x}(t)-A^\top_\lambda(t)x(t) \text{ for }\lambda\in\Lambda,
\end{align}
under the following assumptions:
\begin{enumerate}
\item[$(A1)$]  For each $\lambda\in \Lambda$, the equation $\dot{x}=A_{\lambda}(t)x$ admits an exponential dichotomy on $\R_0^+$ and $\R_0^-$
with respect to the projections $P_{\lambda}^+: \R^d\to \R^d$ and $P_{\lambda}^-:\R^d\to\R^d$, which we assume to be continuous in $\lambda$ according to Theorem \ref{continuous-family-projections}.
\item[$(A2)$]  There is a closed subspace $\emptyset\neq\Lambda_0\subset\Lambda$ such that for each $\lambda\in\Lambda_0$ the equation $\dot{x}=A_{\lambda}(t)x$ has an exponential dichotomy on $\R$.
\end{enumerate}
If not specified otherwise, we always consider the usual scalar products on the Hilbert spaces $L^2(\mathbb{R},\mathbb{R}^d)$ and $H^1(\mathbb{R},\mathbb{R}^d)$. So, in particular, $L_\lambda, L^\ast_\lambda$ are bounded linear operators. The following theorem can be found in \cite[Proposition 3.1]{JacoboAMS} under stronger assumptions. That it actually holds as stated below is explained in \cite[Remark 3.2]{JacoboAMS}; cf. also \cite{Longo}.
\begin{theorem}\label{thm-FredholmIndex}
Under the Assumptions $(A1)$ and $(A2)$ the operators $L_\lambda$ and $L^\ast_\lambda$, $\lambda\in\Lambda$, are Fredholm operators and
\begin{align}\label{index}
\ind L^{\ast}_{\lambda}=\ind L_{\lambda}=\dim N(L_{\lambda})-\dim N(L_{\lambda}^{\ast})
=\dim R(P_{\lambda}^+)-\dim R(P_{\lambda}^-)=0.
\end{align}
\end{theorem}
Above we have called $L^\ast_\lambda$ the adjoint operator of $L_\lambda$, which we now want to justify. Thus we consider for a moment $L_\lambda$ and $L^\ast_\lambda$ as densely defined unbounded operators
\[
L_\lambda:\mathcal{D}(L_\lambda)\subset L^2(\mathbb{R},\mathbb{R}^d)\rightarrow L^2(\mathbb{R},\mathbb{R}^d)
\]
and
\[
L^\ast_\lambda:\mathcal{D}(L^\ast_\lambda)\subset L^2(\mathbb{R},\mathbb{R}^d)\rightarrow L^2(\mathbb{R},\mathbb{R}^d),
\]
where $\mathcal{D}(L_\lambda)=\mathcal{D}(L^\ast_\lambda)=H^1(\mathbb{R},\mathbb{R}^d)$. The following result in operator theory is probably well-known, but we are not aware of any reference in the literature.

\begin{lemma}\label{lemma-adjointcrit}
Let $H$ be a Hilbert space and $T:\mathcal{D}(T)\subset H\rightarrow H$, $S:\mathcal{D}(S)\subset H\rightarrow H$ densely defined operators. If
\begin{enumerate}[$(i)$]
 \item $\langle Tu,v\rangle=\langle u,Sv\rangle$ for all $u\in\mathcal{D}(T)$, $v\in\mathcal{D}(S)$,
 \item $T$ and $S$ are Fredholm operators of index $0$,
 \item $\dim N(T)=\dim N(S)$,
\end{enumerate}
then $S=T^\ast$.
\end{lemma}
\begin{proof}
Let us first recall that the adjoint $T^{*}$ of $T$ is defined by
\[
D(T^{*})=\{v\in H:\ \exists\, w\in H \ \text{such that}\ \langle Tu,v\rangle=\langle u,w\rangle\ \forall u\in D(T)\},
\qquad T^{*}v:=w.
\]
Assumption \textup{(i)} yields for every $v\in D(S)$ the identity
\[
\langle Tu,v\rangle=\langle u,Sv\rangle \quad \forall u\in D(T),
\]
which shows that $v\in D(T^{*})$ and $T^{*}v=Sv$. Thus $S\subset T^{*}$ and it remains to show the opposite inclusion.\\
Clearly, $(i)$ implies that $N(T)\subset R(S)^\perp$. As $T,S$ are Fredholm operators of the same index,
\[
\dim N(T)-\dim R(T)^\perp=\dim N(S)-\dim R(S)^\perp
\]
and so by $(iii)$
\[0=\dim N(T)-\dim N(S)=\dim R(T)^\perp-\dim R(S)^\perp.\]
Thus, as the Fredholm index is $0$,
\[
\dim R(S)^\perp=\dim R(T)^\perp=\dim N(T)
\]
and we see for later reference that under the given assumptions
\begin{align}\label{kerimTSperp}
N(T)=R(S)^\perp.
\end{align}
Let now $u\in\mathcal{D}(T^\ast)$ and set $w=T^\ast u$. Then
\begin{align}\label{langlerangle}
\langle u,Tv\rangle=\langle w,v\rangle,\quad v\in\mathcal{D}(T).
\end{align}
As $R(S)$ is closed, \eqref{kerimTSperp} implies that there are $w_1\in N(T)$ and $u_1\in\mathcal{D}(S)$ such that $w=w_1+Su_1$. It follows from \eqref{langlerangle} and $(i)$ that for $v\in\mathcal{D}(T)$
\begin{align}\label{langlerangleII}
\langle u-u_1,Tv\rangle=\langle u,Tv\rangle-\langle u_1,Tv\rangle=\langle w,v\rangle-\langle Su_1,v\rangle=\langle w-Su_1,v\rangle=\langle w_1,v\rangle.
\end{align}
Once again as $R(S)$ is closed, any $v\in\mathcal{D}(T)$ can be written by \eqref{kerimTSperp} as $v=v_1+v_2$ for some $v_1\in N(T)$ and $v_2\in R(S)$. Actually, $v_2\in R(S)\cap\mathcal{D}(T)$ as $v,v_1\in\mathcal{D}(T)$, and thus we obtain from \eqref{langlerangleII}
\begin{align*}
\langle u-u_1,Tv\rangle=\langle u-u_1,Tv_2\rangle=\langle w_1,v_2\rangle=0,
\end{align*}
where we have used \eqref{kerimTSperp} once again. Consequently, $u-u_1\in R(T)^\perp=N(S)\subset\mathcal{D}(S)$ which shows $u\in\mathcal{D}(S)$ since $u_1\in\mathcal{D}(S)$. Note that here we have swapped $T$ and $S$ in \eqref{kerimTSperp}, which is possible as both operators satisfy identical assumptions.\\
Hence we have shown $\mathcal{D}(T^\ast)\subset\mathcal{D}(S)$, which finally implies $S=T^\ast$ as claimed.
\end{proof}
Now it follows from \eqref{index} and integration by parts that $L^\ast_\lambda$ indeed is the Hilbert space adjoint of $L_\lambda$ when these operators are considered as unbounded operators on $L^2(\mathbb{R},\mathbb{R}^d)$. This yields important information as, e.g., that the kernel of $L^\ast_\lambda$ is the orthogonal complement of the range of the Fredholm operator $L_\lambda$, which will be needed in the proof of our main theorem below. Let us point out that we will make use of the abstract Lemma \ref{lemma-adjointcrit} below once again.
\subsection{Index Theorem and Corollaries}
We consider the family of differential equations \eqref{Lin} under the Assumptions $(\mathrm{A}1)$ and $(\mathrm{A}2)$, where as before $\Lambda$ is a compact and connected topological space. By our discussion in Section \ref{section-expdich}, the corresponding families $E^s_\lambda(0)$, $E^u_\lambda(0)$ for $\lambda\in\Lambda$ yield vector bundles $E^s(0)$ and $E^u(0)$ over the parameter space $\Lambda$ that are both subbundles of the product bundle $\Theta(\mathbb{R}^d)$. Thus, according to our explanations in the beginning of Section \ref{section-vb}, the direct sum  $E^u(0)\oplus E^s(0)$ is a subbundle of $\Theta(\mathbb{R}^{2d})$. Now, the latter bundle can be mapped to $\Theta(\mathbb{R}^d)$ by the bundle homomorphism $(\lambda,u,v)\mapsto (\lambda,u-v)$ for $u,v\in\mathbb{R}^d$, and thus we obtain by restriction a bundle morphism $\mathcal{L}:E^u(0)\oplus E^s(0)\rightarrow\Theta(\mathbb{R}^d)$ defined by $\mathcal{L}(\lambda,u,v)=(\lambda,u-v)$. Note that, as $\dim(E^u(0)\oplus E^s(0))=d$ by \eqref{dimform}, $\mathcal{L}_\lambda$ is an isomorphism if and only if $E^u_\lambda(0)\cap E^s_\lambda(0)=\{0\}\subset\mathbb{R}^d$, which is the case if and only if the operator $L_\lambda$ in \eqref{L} is injective. As $L_\lambda$ is Fredholm of index $0$, this is equivalent to $L_\lambda$ being an isomorphism.\\
Note that Assumption $(A2)$ implies by Theorem \ref{thm-invertible} that
$E^u_\lambda(0)\cap E^s_\lambda(0)=\{0\}\subset\mathbb{R}^d$ for all $\lambda\in\Lambda_0$, which hence shows that the Fredholm operator $L_\lambda$ as well as the bundle morphism $\mathcal{L}_\lambda$ are isomorphisms for all these $\lambda$. Now we can state the main result of this article.
\begin{theorem}\label{thm-main-lin}
If Assumptions $(\mathrm{A}1)$ and $(\mathrm{A}2)$ hold, then
\begin{align}\label{indexformula}
\ind(L)=[E^u(0)\oplus E^s(0),\Theta(\mathbb{R}^d),\mathcal{L}]\in \KO(\Lambda,\Lambda_0),
\end{align}
where the homomorphism $\mathcal{L}:E^u(0)\oplus E^s(0)\rightarrow\Theta(\mathbb{R}^d)$ is defined by $\mathcal{L}(\lambda,u,v)=(\lambda,u-v)$.
\end{theorem}
Before we prove this theorem, let us note a couple of corollaries for later reference. For the first one, we need the following additional assumption.
\begin{itemize}
 \item[(A3)] There are continuous maps $s_1,\ldots,s_d:\Lambda\rightarrow\R^d$ such that
 \[E^u_\lambda(0)=\spann\{s_1(\lambda),\ldots,s_k(\lambda)\},\quad E^s_\lambda(0)=\spann\{s_{k+1}(\lambda),\ldots,s_d(\lambda)\},\quad\lambda\in\Lambda,\]
 for some $1\leq k\leq d-1$. 
\end{itemize} 
Note that $(A3)$ is equivalent to the assumption that the vector bundles $E^u(0)$ and $E^s(0)$ over $\Lambda$ are trivial, i.e., there are vector bundle isomorphisms $\varphi^u:E^u(0)\rightarrow\Theta(\R^k)$ and $\varphi^s:E^s(0)\rightarrow\Theta(\R^{d-k})$. 
\begin{corollary}\label{cor-main-triv}
If $(A1)$, $(A2)$ and $(A3)$ hold, then 
\begin{align*}
\ind(L)=[\Theta(\R^d),\Theta(\R^d),\mathcal{M}]\in \KO(\Lambda,\Lambda_0),
\end{align*}
where $\mathcal{M}:\Lambda\rightarrow\Mat(d,\R)$ is the matrix family $\mathcal{M}_\lambda=(s_1(\lambda),\ldots,s_d(\lambda))$.
\end{corollary} 
\begin{proof}
The maps $s_1,\ldots,s_d:\Lambda\rightarrow\R^d$ in $(A3)$ induce a bundle isomorphism $\varphi:\Theta(\R^d)\rightarrow E^u(0)\oplus E^s(0)$ by $\varphi(\lambda,e_i)=s_i(\lambda)$, $1\leq i\leq k$ and $\varphi(\lambda,e_i)=-s_i(\lambda)$, $k+1\leq i\leq d$  where $e_i$ denotes the i-th standard basis vector in $\R^d$. Now by Theorem \ref{thm-main-lin} and \eqref{ISO1},
\[\ind(L)=[E^u(0)\oplus E^s(0),\Theta(\mathbb{R}^d),\mathcal{L}]=[\Theta(\R^d),\Theta(\mathbb{R}^d),\mathcal{L}\circ\varphi]\in \KO(\Lambda,\Lambda_0)\]
and $(\mathcal{L}\circ\varphi)_\lambda=\mathcal{M}_\lambda$, $\lambda\in\Lambda$.
\end{proof}
In what follows, we denote under assumption $(A3)$ by $\mathcal{L}_D:\Lambda\rightarrow\R$ the map 

\[\mathcal{L}_D(\lambda)=\det(\mathcal{M}_\lambda)=\det(s_1(\lambda),\ldots, s_d(\lambda)).\]
Note that $\mathcal{L}_D(\lambda)=0$ if and only if \eqref{lin-Hamilton} has a non-trivial solution, which reminds of determinant sections for families of Dirac operators (cf. \cite[\S 9.7]{Getzler}).\\
The assumption $(A3)$ in particular holds if $\Lambda$ is a contractible topological space as in this case all vector bundles over $\Lambda$ are trivial. If in addition $\Lambda_0=\{\lambda_0,\lambda_1\}$ consists of two distinct points in $\Lambda$, we obtain $\ind(L)$ as element of $\mathbb{Z}_2$ under the isomorphism $\psi_{\lambda_0,\lambda_1}$ from Proposition \ref{beta-iso}.
\begin{corollary}\label{cor-main-contractible}
Let $\Lambda$ be contractible and $\Lambda_0=\{\lambda_0,\lambda_1\}$ for two distinct points $\lambda_0, \lambda_1\in\Lambda$. If $(A1)$ and $(A2)$ hold, then $\ind(L)\in \KO(\Lambda,\Lambda_0)\cong\mathbb{Z}_2$ is non-trivial if and only if
\[\mathcal{L}_D(\lambda_0)\cdot\mathcal{L}_D(\lambda_1)<0.\]
\end{corollary}   
For the other corollaries, we focus again on the case that the bundles $E^u(0)$ and $E^s(0)$ are possibly non-trivial, i.e., we do no longer explicitly assume that $(A3)$ holds.\\
It follows from \eqref{ISO-S-U} that for any $\tau\geq 0$ there is a bundle isomorphism
\[
\mathcal{X}_\tau:E^u(-\tau) \oplus E^s(\tau) \rightarrow E^u(0)\oplus E^s(0)
\]
given by
\[
\mathcal{X}_\tau(\lambda,u,v)=(\lambda,\mathbb{X}_\lambda(0,-\tau)u,\mathbb{X}_\lambda(0,\tau)v)
\]
where $\mathbb{X} \colon \Lambda \times \mathbb{R}^2 \to \mathrm{GL}(d,\R)$ is as introduced at the end of Section \ref{section-expdich}. This yields the following simple reformulation of Theorem \ref{thm-main-lin}.
\begin{corollary}\label{cor-main-lin}
If Assumptions $(\mathrm{A}1)$ and $(\mathrm{A}2)$ hold, then for any $\tau\geq 0$
\begin{align*}
\ind(L)=[E^u(-\tau)\oplus E^s(\tau),\Theta(\mathbb{R}^d),\mathcal{L}_\tau]\in \KO(\Lambda,\Lambda_0),
\end{align*}
where the homomorphism $\mathcal{L}_\tau:E^u(-\tau)\oplus E^s(\tau)\rightarrow\Theta(\mathbb{R}^d)$ is given by
\[(\mathcal{L}_\tau)(\lambda,u,v)=\mathbb{X}_\lambda(0,-\tau)u-\mathbb{X}_\lambda(0,\tau)v.\]
\end{corollary}
\begin{proof}
As $\mathcal{X}_\tau:E^u(-\tau) \oplus E^s(\tau) \rightarrow E^u(0)\oplus E^s(0)$ is an isomorphism,
\[
[E^u(-\tau) \oplus E^s(\tau), E^u(0)\oplus E^s(0),\mathcal{X}_\tau]=0\in \KO(\Lambda,\Lambda_0).
\]
Thus it follows from Theorem \ref{thm-main-lin} and \eqref{sum} that
\begin{align*}
\ind(L)&=[E^u(0)\oplus E^s(0),\Theta(\mathbb{R}^d),\mathcal{L}]+[E^u(-\tau) \oplus E^s(\tau), E^u(0)\oplus E^s(0),\mathcal{X}_\tau]\\
&=[E^u(-\tau) \oplus E^s(\tau),\Theta(\mathbb{R}^d),\mathcal{L}\circ\mathcal{X}_\tau],
\end{align*}
which shows the claim as $\mathcal{L}\circ\mathcal{X}_\tau=\mathcal{L}_\tau$.
\end{proof}
We now finally discuss the case that the matrix family $A:\Lambda\times\mathbb{R}\rightarrow\Mat(d,\R)$ is asymptotically hyperbolic:
\begin{itemize}
\item[(A4)] There are continuous families $A^\pm:\Lambda\rightarrow\Mat(d,\mathbb{R})$ of hyperbolic matrices such that
\[A^\pm_\lambda:=\lim_{t\rightarrow\pm\infty}A_\lambda(t).\]
\end{itemize}
According to our discussion in Section \ref{section-expdich}, (A4) implies that \eqref{LLL} has an exponential dichotomy on both half axis and thus Assumption (A1) holds. Secondly, we also consider for a non-empty closed subspace $\Lambda_0\subset\Lambda$:
\begin{itemize}
\item[(A5)] For all $\lambda\in\Lambda_0$,
\[E^u_\lambda(0)\oplus E^s_\lambda(0)=\mathbb{R}^d.\]
\end{itemize}
Note that by Theorems \ref{roughness} and \ref{thm-invertible}, the Assumptions $(A4)$ and $(A5)$ imply $(A1)$ and $(A2)$. Thus we obtain from our main Theorem \ref{thm-main-lin} and Corollary \ref{cor-main-lin} the following corollary for asymptotically hyperbolic systems.
\begin{corollary}\label{cor-main-linII}
If Assumptions $(\mathrm{A}4)$ and $(\mathrm{A}5)$ hold, then for any $\tau\geq 0$
\begin{align*}
\ind(L)=[E^u(0)\oplus E^s(0),\Theta(\mathbb{R}^d),\mathcal{L}]
=[E^u(-\tau)\oplus E^s(\tau),\Theta(\mathbb{R}^d),\mathcal{L}_\tau]\in \KO(\Lambda,\Lambda_0),
\end{align*}
where the homomorphisms $\mathcal{L}$ and $\mathcal{L}_\tau$ are defined as above.
\end{corollary}


\subsection{Comparison to the Work of Pejsachowicz and Hu-Portaluri}\label{section-singlepar}
The aim of this subsection is to deduce the main results of the articles \cite{JacoboAMS} and \cite{HuP} from the above Corollary \ref{cor-main-linII} for asymptotically hyperbolic systems.
\subsubsection{Pejsachowicz' Index Formula}
Pejsachowicz considered in \cite{JacoboAMS} the case of the circle $\Lambda=S^1$ as parameter space and $\Lambda_0=\{\lambda_0\}$ for some $\lambda_0\in S^1$ under the Assumptions $(A4)$ and $(A5)$. Let us firstly point out that instead of $(A5)$ it actually is assumed in \cite{JacoboAMS} that both equations \eqref{P-L} and \eqref{A-P-L} only have the trivial homoclinic solution $u\equiv 0$ for $\lambda=\lambda_0$. As $E^u_\lambda(0)\cap E^s_\lambda(0)$ is isomorphic to the space of homoclinic solutions of \eqref{P-L} and, as recalled in Section \ref{section-expdich}, the stable and unstable spaces of the adjoint equation \eqref{A-P-L} are the orthogonal complements of $E^{u/s}_\lambda(0)$ in $\mathbb{R}^d$, this is equivalent to $(A5)$.\\
The aim of this section is to show that our Corollary \ref{cor-main-linII} for $\Lambda=S^1$, $\Lambda_0=\{\lambda_0\}$, i.e.,
\[
\ind(L)=[E^u(0)\oplus E^s(0),\Theta(\mathbb{R}^d),\mathcal{L}]\in \KO(S^1,\{\lambda_0\})
\]
is equivalent to Pejsachowicz' Index Theorem \cite{JacoboAMS}. Before we can state the exact result, we need some preliminaries.\\
Firstly, Pejsachowicz considers the index bundle
\begin{align}\label{indPejsachowicz}
[E(L,V)]-[\Theta(V)]\in\widetilde{\KO}(S^1)
\end{align}
as introduced by Atiyah and J\"anich for families $L:S^1\rightarrow\Phi_0(X,Y)$, which by \eqref{IndBundClassical} is $\Psi_{\lambda_0}(\ind(L))$ for the isomorphism $\Psi_{\lambda_0}:\KO(S^1,\{\lambda_0\})\rightarrow\widetilde{\KO}(S^1)$ introduced in \eqref{alpha}. Secondly, let us consider a family of autonomous differential equations of the form
\[
\dot{x}=B_\lambda x
\]
for some family $B:S^1\rightarrow\Mat(d,\mathbb{R})$ of hyperbolic matrices $B_\lambda$, i.e., all eigenvalues of $B_\lambda$ have non-vanishing real parts. Then the stable and unstable spaces $E^s_\lambda(0)$, $E^u_\lambda(0)$ are the generalised eigenspaces with respect to eigenvalues having negative or positive real part, respectively. As these are the images of the spectral projections regarding the eigenvalues on the negative or positive half-plane, and as the spectral projections depend continuously on the matrices, we obtain two vector bundles $E^u_B$ and $E^s_B$ over $S^1$. Note that $E^u_B$ and $E^s_B$ can also be obtained from Theorem \ref{continuous-family-projections}, but here we want to point out that they can be constructed without the concept of exponential dichotomies following \cite{JacoboAMS}.\\
Now we are ready to show that Pejsachowicz' index theorem is equivalent to Corollary \ref{cor-main-linII} in the special case that $\Lambda=S^1$ and $\Lambda_0$ is a singleton.

\begin{corollary}\label{cor-Jacobo}
If $(A4)$ and $(A5)$ hold for the family \eqref{Hamilton} parameterised by $\Lambda=S^1$ and for $\Lambda_0=\{\lambda_0\}$, then
\[\Psi_{\lambda_0}(\ind(L))=[E^s_{A^+}]-[E^s_{A^-}]\in\widetilde{\KO}(S^1)\cong\mathbb{Z}_2.\]
\end{corollary}
\begin{proof}
By $(A4)$, $(A5)$ and Theorem \ref{roughness} the equation $\dot{x}=A_{\lambda_0}(t)x$ has an exponential dichotomy on all of $\mathbb{R}$, and henceforth we let $K\geq 1$ and $\alpha>0$ be the corresponding constants in Definition \ref{def-expdich}. We now let $\tau_0>0$ be sufficiently large such that
\[
\sup_{t\in[\tau_0,\infty)}|A_{\lambda_0}(t)-A^+_{\lambda_0}|<\frac{\alpha}{4K^2},\qquad \sup_{t\in(-\infty,-\tau_0]}|A_{\lambda_0}(t)-A^-_{\lambda_0}|<\frac{\alpha}{4K^2}.
\]
Moreover, we let $\R \times S^1 \ni (t,\lambda) \mapsto C_\lambda(t) \in \Mat(d,\R)$ be a continuous matrix family such that
\begin{align}\label{estimatePej}
\sup_{t\in\R}|A_{\lambda_0}(t)-C_{\lambda_0}(t)|<\frac{\alpha}{4K^2},
\end{align}
and $C_\lambda(t)=A^-_\lambda$ for $t\leq -\tau_0$, as well as $C_\lambda(t)=A^+_\lambda$ for $t\geq \tau_0$. If we now consider
\[
K_\lambda(t):=A_\lambda(t)-C_\lambda(t),
\]
then the operator $S_{K_{\lambda}}:H^1(\R,\R^p)\rightarrow L^2(\R,\R^p)$, $(S_{K_{\lambda}}x)(t)=K_{\lambda}(t)x(t)$ is compact by \cite[Lemma 3.3]{AlbertoMorseHilbert} as $\lim\limits_{t\to\pm\infty}K_{\lambda}(t)=0$. Now we set
\[
L_{\lambda}=\widetilde{L}_{\lambda}-S_{K_{\lambda}},
\]
where $L_{\lambda}x=\dot{x}-A_{\lambda}x$ and $\widetilde{L}_{\lambda}x=\dot{x}-C_{\lambda}x$. The operators $\widetilde{L}_{\lambda_0}-sS_{K_{\lambda_0}}$ are invertible for all $s\in[0,1]$ by Theorem \ref{roughness}, \eqref{estimatePej} and Theorem \ref{thm-invertible}.\\
It follows from the homotopy invariance of the index bundle in Proposition \ref{index-bundle} and Corollary \ref{cor-main-linII} that
\begin{equation*}
\ind(L)=\ind(\widetilde{L}+S_K)=\ind(\widetilde{L})=\left[ E^u_{A^-} \oplus E^s_{A^+},\, \Theta(\R^d),\, \mathcal{L}_{\tau_0} \right] ,
\end{equation*}
where we have used in the last equality that $E^u_\lambda(-\tau_0)=E^u_{A^-}$ and $E^s_\lambda(\tau_0)=E^s_{A^+}$, $\lambda\in S^1$, for the equation $\dot{x}=C_\lambda x$. Hence we obtain in $\widetilde{KO}(S^1)$

\begin{align*}
\Psi_{\lambda_0}(\ind(L))&=[ E^u_{A^-} \oplus E^s_{A^+}]-[\Theta(\mathbb{R}^d)]=[E^s_{A^-}\oplus  E^u_{A^-} \oplus E^s_{A^+}]-[E^s_{A^-}\oplus\Theta(\mathbb{R}^d)]\\
&=[\Theta(\mathbb{R}^d)\oplus E^s_{A^+}]-[E^s_{A^-}\oplus\Theta(\mathbb{R}^d)]=[E^s_{A^+}]-[E^s_{A^-}]\in\widetilde{KO}(S^1), 
\end{align*}
where we have used that $[E]-[F]=0\in\widetilde{KO}(S^1)$ if the vector bundles $E$ and $F$ are isomorphic, as well as the fact that $E\oplus F$ is isomorphic to $F\oplus E$.

\end{proof}
Let us finally point out that Pejsachowicz generalised his index theorem in \cite{JacoboII} to the case that $\Lambda$ is a general compact topological space and $\Lambda_0=\{\lambda_0\}$ a singleton. It turns out that the above formula for $\ind(L)$ verbatim holds when we replace $\widetilde{\KO}(S^1)$ by $\widetilde{\KO}(\Lambda)$, which again follows from our main Theorem \ref{thm-main-lin} as well by the same proof as in Corollary \ref{cor-Jacobo}.




\subsubsection{Hu and Portaluri's Index Formula}
Hu and Portaluri considered in \cite{HuP} the case $\Lambda=[a,b]$, $\Lambda_0=\{a,b\}$ under the Assumptions $(A4)$ and $(A5)$, and constructed two $\mathbb{Z}_2$-valued invariants for \eqref{LLL}, which they equated in their main theorem. Let us first introduce the $\mathbb{Z}_2$-index of \eqref{LLL} in the terminology of \cite{HuP}. As the families of stable and unstable spaces $\{E^u_\lambda(0)\}_{\lambda\in [a,b]}$ and $\{E^s_\lambda(0)\}_{\lambda\in [a,b]}$ are continuous paths in the Grassmannians $Gr_k(d,\mathbb{R})$ and $Gr_{d-k}(d,\mathbb{R})$ for some $1\leq k\leq d-1$ by (A5), there are frames $\{v_1(\lambda),\ldots,v_k(\lambda)\}_{\lambda\in [a,b]}$ for $\{E^u_\lambda(0)\}_{\lambda\in [a,b]}$ and $\{w_1(\lambda),\ldots,w_{d-k}(\lambda)\}_{\lambda\in [a,b]}$ for $\{E^s_\lambda(0)\}_{\lambda\in [a,b]}$. We consider the matrix family $\{M(\lambda)\}_{\lambda\in [a,b]}$, where
\[M(\lambda)=(v_1(\lambda),\ldots, v_k(\lambda),w_1(\lambda),\ldots,w_{d-k}(\lambda))\in\Mat(d,\mathbb{R})\]
and set as in \cite[Definition 2.1]{HuP}
\[\iota(E^u_\lambda(0),E^s_\lambda(0);\lambda\in I)=\begin{cases}
0\quad\text{if}\quad \det(M(0)M(1))>0\\
1\quad\text{if}\quad \det(M(0)M(1))<0.
\end{cases}\]
The construction of the other index in \cite{HuP} is as follows. As in our construction \eqref{defiindbund}, the authors consider the vector bundle $E(L,V)$ over $[a,b]$ and the restriction $L|_{E(L,V)}:E(L,V)\rightarrow\Theta(V)$. As every bundle over $[a,b]$ is trivial, there are bundle isomorphisms $\varphi:E(L,V)\rightarrow\Theta(\R^d)$, $\phi:\Theta(V)\rightarrow\Theta(\R^d)$ and we obtain a morphism $\tilde{L}:=\phi\circ L\circ\varphi^{-1}:\Theta(\mathbb{R}^d)\rightarrow\Theta(\mathbb{R}^d)$, where $d=\dim(V)$. Now $S^1$ can be obtained by gluing two copies of $[a,b]$ at the endpoints and the product bundles can be glued by the clutching function $\tilde{L}$ which is an isomorphism over $\partial[a,b]=\{a,b\}$. The result is a vector bundle $E_L$ over $S^1$ and Hu and Portaluri call
\[
\sigma(L_\lambda,\lambda\in [a,b])=\begin{cases}
0\quad\text{if}\quad E_L\quad\text{is orientable},\\
1\quad\text{if}\quad E_L\quad\text{is non-orientable}
\end{cases}
\]
the parity of the path $L$, which is just the first Stiefel-Whitney number of the vector bundle $E_L$. The main theorem of \cite{HuP} states that
\begin{align}\label{HuPortaluri}
\sigma(L_\lambda,\lambda\in [a,b])=\iota(E^u_\lambda(0),E^s_\lambda(0);\lambda\in [a,b])\in\mathbb{Z}_2,
\end{align}
and we claim that this is an immediate consequence of our Corollary \ref{cor-main-contractible} for $\Lambda=[a,b]$, $\Lambda_0=\{a,b\}$. Indeed, from the latter corollary it only remains to identify $\psi_{a,b}(\ind(L))$ and $\sigma(L_\lambda,\lambda\in [a,b])$ as elements of $\mathbb{Z}_2$. The construction of $E_L$ in the definition of $\sigma(L_\lambda,\lambda\in [a,b])$ is made by a clutching of two trivial bundles on $[a,b]$ along $\{a,b\}$ by $\tilde{L}_\lambda\in\GL(d,\mathbb{R})$, $\lambda\in \{a,b\}$. Now $E_L$ is orientable if and only if $\tilde{L}_a$ and $\tilde{L}_b$ belong to the same path-component of $\GL(d,\mathbb{R})$ (see, e.g., \cite[\S 1.2]{Hatcher2}), which is the case if and only if $\det(\tilde{L}_a)\det(\tilde{L}_b)>0$. Hence it follows from \eqref{defiindbund} and \eqref{beta} that
\[\sigma(L_\lambda,\lambda\in [a,b])=\psi_{a,b}(\ind(L))\in\mathbb{Z}_2,\]
and so \eqref{HuPortaluri} indeed is a simple consequence of Corollary \ref{cor-main-contractible}.


\subsection{Proof of the Index Theorem}
We divide the proof into four major steps.
\subsubsection*{Step I: From $L$ to $L_P$.}
Consider for some fixed $\tau_0>0$ the family of differential operators
\begin{align*}
L^0_\lambda:\mathcal{D}(L^0_\lambda)\subset L^2([-\tau_0,\tau_0],\mathbb{R}^d)\rightarrow L^2([-\tau_0,\tau_0],\mathbb{R}^d),\quad (L^0_\lambda x)(t)=\dot{x}(t)-A_\lambda(t)x(t),
\end{align*}
where
\begin{align*}
\mathcal{D}(L^0_\lambda):=\left\{x\in H^1([-\tau_0,\tau_0],\mathbb{R}^d): x(-\tau_0)\in E^u_\lambda(-\tau_0),\, x(\tau_0)\in E^s_\lambda(\tau_0)\,\right\}.
\end{align*}
These operators are vital for understanding the family $\{L_\lambda\}_{\lambda\in\Lambda}$ and we now first discuss basic analytic properties.
\begin{lemma}\label{Lemma-FredInd0}
The operators $L^0_\lambda$, $\lambda\in\Lambda$, are Fredholm of index $0$.
\end{lemma}
\begin{proof}
To simplify notation we set in this proof $Y=L^2([-\tau_0,\tau_0],\mathbb{R}^d)$, $X=H^1([-\tau_0,\tau_0],\mathbb{R}^d)$ and $Z=\mathcal{D}(L^0_\lambda)$. The operator $L^0_\lambda:Z\rightarrow Y$ is of the form $L^0_\lambda=Q|_Z+K_\lambda |_Z$, where
$Q:X\rightarrow Y$ is defined by $Qx=\dot{x}$, and $K_\lambda:X\rightarrow Y$ is the multiplication operator $(K_\lambda x)(t)=-A_\lambda(t)x(t)$. As $K_\lambda$ extends to a bounded operator on $Y$ and the embedding of $X$ into $Y$ is compact, it follows that $K_\lambda:X\rightarrow Y$ is a compact operator. The operator $Q$ is surjective and its kernel is the $d$-dimensional subspace of constant functions in $X$. Thus $Q$ is a Fredholm operator of index $d$. Now consider the homomorphism
\[
T:X\rightarrow\mathbb{R}^d\oplus\mathbb{R}^d,\quad x\mapsto (P^-_\lambda(-\tau_0) x(-\tau_0),(I_d-P^+_\lambda(\tau_0))x(\tau_0)),
\]
where $P^-_\lambda(-\tau_0)$ and $P^+_\lambda(\tau_0)$ are projections as in Theorem \ref{continuous-family-projections}. Then $N(T)=\mathcal{D}(L^0_\lambda)$ and the codimension of this space in $X$ is by the first isomorphism theorem the dimension of the range of $T$, which is $\dim(E^u_\lambda(-\tau_0)^\perp)+\dim(E^s_\lambda(\tau_0)^\perp)=d$.\\
Finally the claim follows from \cite[XI.3]{Gohberg}, which states that the restriction of a Fredholm operator $D:X\rightarrow Y$ to a closed subspace $Z\subset X$ of finite codimension is a Fredholm operator of index $\ind(D|_Z)=\ind(D)-\codim(Z)$.
\end{proof}
If we consider for $\lambda\in\Lambda$ the unbounded operators
\begin{align*}
(L^0_\lambda)^\ast&:\mathcal{D}((L^0_\lambda)^\ast)\subset L^2([-\tau_0,\tau_0],\mathbb{R}^d)\rightarrow L^2([-\tau_0,\tau_0],\mathbb{R}^d),
\quad ((L^{0}_{\lambda})^{\ast}x)(t)=-\dot{x}(t)-A_\lambda(t)^{\top}x(t),
\end{align*}
where
\begin{equation*}
\mathcal{D}((L^0_\lambda)^\ast)=\{x\in H^1([-\tau_0,\tau_0],\mathbb{R}^d):\, x(-\tau_0)\in E^u_{\lambda}(-\tau_0)^\perp,\, x(\tau_0)\in E^s_{\lambda}(t_0)^\perp\},
\end{equation*}
then we obtain the following lemma, which will be needed below.
\begin{lemma}\label{lemma-adjointL0}
The operator $(L^0_\lambda)^\ast$ is the Hilbert space adjoint of $L^0_\lambda$ in $L^2([-\tau_0,\tau_0],\mathbb{R}^d)$.
\end{lemma}
\begin{proof}
This follows from Lemma \ref{lemma-adjointcrit}. We firstly note that $(i)$ follows from integration by parts. For $(ii)$ we know already from Lemma \ref{Lemma-FredInd0} that $L^0_\lambda$ is a Fredholm operator of index $0$ and the same argument applies to $(L^0_\lambda)^\ast$ as well, when noting that the codimension of $\mathcal{D}((L^0_\lambda)^\ast)$ in $H^1([-\tau_0,\tau_0],\mathbb{R}^d)$ is $\dim(E^u_\lambda(-\tau_0))+\dim(E^s_\lambda(\tau_0))=d$. Finally, $(iii)$ holds as
\[\dim N((L^0_\lambda)^\ast)=\dim N(L^\ast_\lambda)=\dim N(L_\lambda)=\dim N(L^0_\lambda),\]
where the first and the last equality are a consequence of the definition of $\mathcal{D}((L^0_\lambda)^\ast)$ and $\mathcal{D}(L^0_\lambda)$, respectively, and the middle equality is part of Theorem \ref{thm-FredholmIndex}.
\end{proof}
Note that for every $\lambda\in\Lambda$, there is a canonical map
\begin{align}\label{canonical}
i_\lambda: \cD(L_{\lambda}^0) \to H^1(\R,\R^d)
\end{align}
defined by extending a given function \( u \in \cD(L_{\lambda}^0) \) to the intervals $ (-\infty, -\tau_0)$ and $(\tau_0, \infty)$ as a solution of $\dot{x}=A_\lambda(t)x$. The map $i_\lambda$ is injective, and
\begin{equation}\label{injective-kernels}
    i_{\lambda}\left(N(L_{\lambda}^0)\right) = N(L_{\lambda}),
\end{equation}
which in particular implies that $\dim N(L_{\lambda}) =\dim N(L_{\lambda}^0)$. Moreover, there is a commutative diagram
\begin{align}\label{diagram-i}
    \begin{split}
    \xymatrix{
 H^1(\R,\R^d)\ar[r]^-{L_{\lambda}}& L^2(\R,\R^d)\ar[d]^p\\
    \cD(L_{\lambda}^0)\ar[r]^-{L_{\lambda}^0}\ar[u]_{i_{\lambda}}& L^{2}([-\tau_0,\tau_0],\R^d),}
    \end{split}
\end{align}
where $p$ denotes the restriction of functions in $L^2(\R,\R^d) $ to $L^2([-\tau_0,\tau_0],\R^d)$.\\
Let us recall that the index bundle can be defined for operator families having varying domains as explained in the final paragraph of Section \ref{subsec-indexbundle}. In the following lemma, we use the notation introduced in that section.
\begin{lemma}\label{lemma-dombundle}
There is a family of projections $P:\Lambda\rightarrow\mathcal{L}(H^1([-\tau_0,\tau_0],\R^d))$ such that
\begin{equation*}
H^1([-\tau_0,\tau_0],\R^d)_P=\left\{(\lambda,x)\in\Lambda\times H^1([-\tau_0,\tau_0],\R^d):\, x\in\mathcal{D}(L^0_\lambda)\right\},
\end{equation*}
where we use the notation introduced in \eqref{Banach-P}.
\end{lemma}
\begin{proof}
By definition of $H^1([-\tau_0,\tau_0],\R^d)_P$ we need to construct a continuous family of projections
\[ P:\Lambda\rightarrow\mathcal{L}(H^1([-\tau_0,\tau_0],\mathbb{R}^d)) \text{ with }
R(P_\lambda)=\mathcal{D}(L^0_\lambda).
\]
Let $P^\pm:\Lambda\rightarrow\mathcal{L}(\mathbb{R}^d)$ be two families of projections such that $N(P^-_\lambda)=E^u_\lambda(-\tau_0)$ and $R(P^+_\lambda)=E^s_\lambda(\tau_0)$ as in Theorem \ref{continuous-family-projections}. Consider $P: \Lambda\to \mathcal{L}(H^1([-\tau_0,\tau_0],\mathbb{R}^d))$ given by
\begin{equation*}
(P_\lambda x)(t)=x(t)-\frac{\tau_0-t}{2\tau_0}P^-_\lambda x(-\tau_0)-\frac{\tau_0+t}{2\tau_0}(I_d-P^+_\lambda)x(\tau_0).
\end{equation*}
Then
\begin{align*}
(P_{\lambda}x)(-\tau_0)&=x(-\tau_0)-P^-_{\lambda}x(-\tau_0)=(I_d-P_{\lambda}^-)x(-\tau_0)\in E^u_{\lambda}(-\tau_0),\\
(P_{\lambda}x)(\tau_0)&=x(\tau_0)-(x(\tau_0)-P^+_{\lambda}x(\tau_0))=P_{\lambda}^+x(\tau_0)\in E^s_{\lambda}(\tau_0),
\end{align*}
which shows $R(P_\lambda)\subset\mathcal{D}(L^0_\lambda)$ for all $\lambda\in\Lambda$. As $P_\lambda x=x$ for all $x\in\mathcal{D}(L^0_\lambda)$,
it follows that $R(P_\lambda)=\mathcal{D}(L^0_\lambda)$ and $P^2_\lambda=P_\lambda$ for all $\lambda\in\Lambda$.
\end{proof}
Note that we obtain a corresponding bundle morphism
\begin{align*}
L_P:H^1([-\tau_0,\tau_0],\R^d)_P\rightarrow\Theta(L^2([-\tau_0,\tau_0],\mathbb{R}^d)),\quad (L_P)_\lambda=L^0_\lambda
\end{align*}
as the map
\[\Lambda\ni\lambda\mapsto \frac{d}{dt}-A_\lambda(\cdot)\in\mathcal{L}(H^1([-\tau_0,\tau_0],\R^d),L^2([-\tau_0,\tau_0],\R^d))\]
is continuous. Hence in summary
\[L:\Theta(H^1(\mathbb{R},\mathbb{R}^d))\rightarrow\Theta(L^2(\mathbb{R},\mathbb{R}^d))\quad\text{and}\quad L_P:H^1([-\tau_0,\tau_0],\R^d)_P\rightarrow\Theta(L^2(\mathbb{R},\mathbb{R}^d))\]
are Fredholm morphisms of index $0$ that are invertible over $\Lambda_0$. Thus $\ind(L)$ and $\ind(L_P)$ are defined as elements in $KO(\Lambda,\Lambda_0)$, and the aim of the next step is to equate these $K$-theory classes.
\subsubsection*{Step II: $\ind(L)=\ind(L_P)\in \KO(\Lambda,\Lambda_0)$}
For the construction of the index bundle as introduced in Section \ref{section-vb} we first need to find finite dimensional spaces $V \subseteq L^{2}(\R,\R^d)$ and $W\subseteq L^{2}([-\tau_0,\tau_0],\R^d)$ that are transversal to the images of the families $L$ and $L_P$, i.e.,
\begin{align*}
R(L_{\lambda}) + V &= L^{2}(\R,\R^d)\quad \text{ and }\quad R(L^0_{\lambda}) + W = L^{2}([-\tau_0,\tau_0],\R^d) \text{ for all } \lambda \in \Lambda.
\end{align*}
The general existence of such spaces has been stated in Section \ref{section-vb}, however, below we need a special form of them which is tailored to our argument and which we now describe.\\
Let $\lambda_0 \in \Lambda$ be fixed. As $ N(L_{\lambda_0})$ is closed, there exists a closed subspace
$E_{\lambda_0} \subseteq H^1(\R,\R^d)$ such that
\begin{equation*}
N(L_{\lambda_0}) \oplus E_{\lambda_0} = H^1(\R,\R^d).
\end{equation*}
Now consider the bounded linear operators
\begin{equation*}
C_{\lambda} : E_{\lambda_0} \times N(L_{\lambda_0}^{\ast}) \to L^{2}(\R,\R^d),
\end{equation*}
defined by
\begin{equation*}
C_{\lambda}(w, v) := L_{\lambda}w + v.
\end{equation*}
Since $N(L_{\lambda_0}^{\ast})=R(L_{\lambda_0})^\perp$, where we use that $L_{\lambda_0}^{\ast}$ is the Hilbert space adjoint of $L_{\lambda_0}$,
\begin{equation*}
C_{\lambda_0} \in \operatorname{GL}\left(E_{\lambda_0} \times N(L_{\lambda_0}^{\ast}), L^{2}(\R,\R^d)\right).
\end{equation*}
As $L$ is continuous and the set of bounded invertible operators is open, there is a neighbourhood $U_{\lambda_0} \subseteq \Lambda$ of $\lambda_0$ such that
\begin{equation*}
C_{\lambda} \in \operatorname{GL}(E_{\lambda_0} \times N(L_{\lambda_0}^{\ast}), L^2(\R,\R^d))
\text{ for all } \lambda \in U_{\lambda_0}.
\end{equation*}
Consequently,
\begin{equation*}
R(L_{\lambda}) + N(L_{\lambda_0}^{\ast}) = L^{2}(\R,\R^d) \text{ for all } \lambda \in U_{\lambda_0}.
\end{equation*}
Since $\Lambda$ is compact, it can be covered by a finite number of such neighbourhoods
\begin{equation*}\label{cover-U}
U_{\lambda_1}, \ldots, U_{\lambda_n} \subseteq \Lambda,
\end{equation*}
and we now set
\begin{equation}\label{space-V}
V := N(L_{\lambda_1}^{\ast}) + N(L_{\lambda_2}^{\ast}) + \cdots + N(L_{\lambda_n}^{\ast}) \subseteq L^{2}(\R,\R^d).
\end{equation}
Then $\dim V < \infty$, and for all $\lambda \in \Lambda$
\begin{equation}\label{transversalV}
R(L_{\lambda}) + V = L^{2}(\R,\R^d).
\end{equation}
Bearing in mind Lemma \ref{lemma-adjointL0}, we can repeat these arguments and see that
\begin{align}\label{transversalW}
R(L_{\lambda}^0)+ W= L^{2}([-\tau_0,\tau_0],\R^d),
\end{align}
where
\begin{align}\label{space-W}
W:=N((L_{\lambda_1}^0)^\ast)+\cdots +N((L_{\lambda_n}^0)^\ast).
\end{align}
Here we can indeed assume without loss of generality that the parameters $\lambda_1,\ldots,\lambda_n\in \Lambda$ are the same for $V$ and $W$, as we can add any finite number of spaces $N(L^\ast_\lambda)$ to $V$ and any finite number of spaces $N((L_{\lambda}^0)^{\ast})$ to $W$ without affecting the properties \eqref{transversalV} and \eqref{transversalW}.\\
Now the vector bundles
	\begin{align*}
		E(L,V)&:=\set{(\lambda,x)\in \Lambda\tm H^1(\R,\R^d)\mid L_{\lambda}x\in V},\\
		E(L_P,W)&:=\set{(\lambda,x)\in \Lambda\tm \cD(L_{\lambda}^0)\mid L_{\lambda}^0x\in W}
	\end{align*}
are defined and
\begin{align*}
\ind(L)&=[E(L,V),V,L|_{E(L,V)}]\in \KO(\Lambda,\Lambda_0),\\
\ind(L_P)&=[E(L_P,W),W,L_P|_{E(L_P,W)}]\in \KO(\Lambda,\Lambda_0).
\end{align*}
The particular choice of the spaces $V$ and $W$ now yields the following observation. Let us denote as in \eqref{diagram-i} by $p:L^2(\mathbb{R},\mathbb{R}^d)\rightarrow L^2([-\tau_0,\tau_0],\mathbb{R}^d)$ the restriction to the interval $[-\tau_0,\tau_0]$. We consider
\begin{equation*}
V_0:=\{\chi_{[-\tau_0,\tau_0]}u\mid\, u\in V\,\},\qquad V_1:=\{\chi_{\R\setminus[-\tau_0,\tau_0]}u\mid u\in V\,\},
\end{equation*}
and note that $V_0\cap V_1=\{0\}$ as well as $V\subset V_0\oplus V_1$. Moreover, the definition of the spaces $V$ and $W$ in \eqref{space-V}, \eqref{space-W} show that $p(V_0\oplus V_1)\subset W$, and the restriction $p\mid_{V_0}:V_0\rightarrow W$ is an isomorphism. We obtain from \eqref{diagram-i} a commutative diagram
\begin{align}\label{commdiag1}
\begin{split}
\xymatrix{
E(L,V_0\oplus V_1)_{\lambda}\ar[r]^-{L_{\lambda}}& V_0\oplus V_1\ar[d]^p\\
E(L_P,W)_{\lambda}\ar[u]_{(\iota_E)_{\lambda}}\ar[r]^-{L_{\lambda}^0}& W,
}
\end{split}
\end{align}
where $\iota_E:E(L_P,W)\rightarrow E(L,V_0\oplus V_1)$ is the bundle morphism induced by the canonical maps \eqref{canonical}.
Since $i_E:E(L_P,W)\to E(L,V_0\oplus V_1)$ is an injective bundle morphism, it follows that $E^0:=\iota_E(E(L_P,W))$ is a subbundle of $E(L,V_0\oplus V_1)$. Let $E^1$ be a complementary subbundle to $E^0$, i.e., $E(L,V_0\oplus V_1)=E^0\oplus E^1$. Now, for any \( \lambda \in \Lambda \), the diagram \eqref{commdiag1} is of the form
\begin{align*}
\begin{split}
\xymatrix{
E^0_{\lambda}\oplus E^1_{\lambda}\ar[r]^-{L_{\lambda}}& V_0\oplus V_1\ar[d]^p\\
E(L_P,W)_{\lambda}\ar[u]_{(i_E)_{\lambda}}\ar[r]^-{L_{\lambda}^0}& W
}
\end{split}
\end{align*}
and $L_{\lambda}: E^0_{\lambda}\oplus E^1_{\lambda}\to V_0\oplus V_1$ can be written as operator matrix
	$$
		L_{\lambda}
		=
		\begin{bmatrix}
			L_{\lambda}^{11}& L_{\lambda}^{12}\\
& \\
			L_{\lambda}^{21}& L_{\lambda}^{22}
		\end{bmatrix}.
	$$
Note that in the following lemma $(i)$ and $(iii)$ concern the operators on the whole space $\Lambda$, whereas $(ii)$ only holds on $\Lambda_0$.
\begin{lemma}
The above matrix elements of $L$ have the following properties:
\begin{itemize}
\item[$(i)$] The homomorphism $L_{\lambda}^{21}:E^0_{\lambda}\to V_1$ is trivial for all $\lambda\in\Lambda$.
\item[$(ii)$] $L_{\lambda}^{11}:E^0_{\lambda}\to V_0$ is an isomorphism for all $\lambda\in\Lambda_0$.
\item[$(iii)$] $L_{\lambda}^{22}:E^1_{\lambda}\to V_1$ is an isomorphism for all $\lambda\in\Lambda$.
\end{itemize}
\end{lemma}
\begin{proof}
We firstly consider $L_{\lambda}^{21}$ and let $u\in E^0_{\lambda}$. Since there exists $v\in E(L_P,W)_{\lambda}$ with $(i_E)_{\lambda}v=u$, it follows that $L_{\lambda}u$ has the property
\begin{equation*}
(L_{\lambda}u)(t)=0 \text{ for all } t\in (-\infty,-\tau_0)\cup (\tau_0,\infty),
\end{equation*}
which yields
\begin{equation*}
L_{\lambda}u=(L_{\lambda}^{11}u,L_{\lambda}^{21}u)\in V_0\oplus\{0\},
\end{equation*}
and therefore the operator $L_{\lambda}^{21}$ is trivial.\\
To show $(ii)$, we just need to note that $L_{\lambda}^{11}:E^0_\lambda\rightarrow V_0$ is injective for $\lambda\in\Lambda_0$ by $(i)$, as otherwise $\ker(L_\lambda)\neq \{0\}$ which contradicts that $L_\lambda$ is invertible for $\lambda\in\Lambda_0$. Now $\dim(V_0)=\dim(W)$ and the latter is equal to $\dim(E(L_P,W)_{\lambda})$ by \eqref{ker-I-P}. As $\dim(E(L_P,W)_{\lambda})=\dim (i_E)_{\lambda}(E(L_P,W)_{\lambda})=\dim(E^0_\lambda)$, we have $\dim(E^0_\lambda)=\dim(V_0)$ and thus $L_{\lambda}^{11}:E^0_\lambda\rightarrow V_0$ indeed is an isomorphism if $\lambda\in\Lambda_0$.\\
Finally, to see that $L_{\lambda}^{22}$ is an isomorphism, let us consider some $u\in E^1_\lambda\subset H^1(\mathbb{R},\mathbb{R}^d)$ such that $L^{22}_\lambda u=0$. Then
\begin{equation*}
\dot{u}(t)-A_\lambda(t)u(t)=0 \text{ for all } t\in(-\infty,-\tau_0) \text{ and for } t\in(\tau_0,\infty),
\end{equation*}
which particularly shows that
\begin{equation*}
u(-\tau_0)\in N(P_\lambda^-(-\tau_0)) \text{ and } u(\tau_0)\in R(P_\lambda^+(\tau_0)).
\end{equation*}
Hence, there exists an element $v\in E(L_P,W)_\lambda$ such that $(i_E)_\lambda(v)=u$, which implies that $u\in E^0_{\lambda}$. Consequently, $u\in E^0_\lambda\cap E^1_\lambda=\{0\}$ and thus $L_{\lambda}^{22}$ is injective. Moreover, as $\dim(V_0\oplus V_1)=\dim(E^0_{\lambda}\oplus E^1_{\lambda})$ by \eqref{ker-I-P} and $\dim(V_0)=\dim(E^0_\lambda)$ by the proof of $(ii)$, we have $\dim(V_1)=\dim(E^1_\lambda)$ and so $L^{22}_\lambda$ is an isomorphism for all $\lambda\in\Lambda$.
\end{proof}
By the previous lemma, we can now argue as in \cite[\S 3.3, Step 4]{NilsGeod} and note that the homotopy
\begin{equation*}\label{decomp-S}
h(s,\lambda)=\begin{bmatrix}
			L_{\lambda}^{11}& (1-s)\,L_{\lambda}^{12}\\
& \\
			0& L_{\lambda}^{22}
		\end{bmatrix}
\end{equation*}
has the property that $h(s,\lambda):E(L,V_0\oplus V_1)_\lambda\rightarrow V_0\oplus V_1$ is invertible for all $s\in[0,1]$ and $\lambda\in\Lambda_0$. Thus by \eqref{homotopy}
\begin{align*}
\ind(L)&=[E(L,V_0\oplus V_1),\Theta(V_0\oplus V_1),h(0,\cdot)]=[E(L,V_0\oplus V_1),\Theta(V_0\oplus V_1),h(1,\cdot)]\\
&=[E(L,V_0\oplus V_1),\Theta(V_0\oplus V_1),L^{11}\oplus L^{22}
]\\
&=[E^0\oplus E^1,\Theta(V_0)\oplus\Theta(V_1),L^{11}\oplus L^{22}]\in \KO(\Lambda,\Lambda_0).
\end{align*}
As $L^{22}:E^1\rightarrow\Theta(V_1)$ is a bundle isomorphism, it follows that
\begin{align*}
[E^0\oplus E^1,\Theta(V_0)\oplus\Theta(V_1),L^{11}\oplus L^{22}]&=[E^0,\Theta(V_0),L^{11}]+[E^1,\Theta(V_1), L^{22}]\\
&=[E^0,\Theta(V_0),L^{11}].
\end{align*}
Finally, the commutative diagram
	\begin{equation*}
		\xymatrix{
		E^0_{\lambda}\ar[r]^-{L_{\lambda}^{11}}& V_0\ar[d]^{p|_{V_0}}_{\cong}\\
		E(L_P,W)_{\lambda}\ar[u]^{(i_E)_{\lambda}}_{\cong}\ar[r]^-{L^0_{\lambda}}& W
		}
	\end{equation*}
yields by \eqref{ISO1}

\begin{align*}
\begin{split}
[E^0,\Theta(V_0),L^{11}]=[E(L_P,W),\Theta(W),L_P]\in \KO(\Lambda,\Lambda_0)
\end{split}
\end{align*}
and we indeed have shown that $\ind(L)=\ind(L_P)\in \KO(\Lambda,\Lambda_0)$ as claimed.
\subsubsection*{Step III: $\ind(L_P)=\ind(Q)\in \KO(\Lambda,\Lambda_0)$}
The aim of this step of the proof is to show that
\begin{equation*}\label{Step-3}
\ind(L_P)=\ind(Q)\in \KO(\Lambda,\Lambda_0),
\end{equation*}
where $Q=\{Q_\lambda\}_{\lambda\in\Lambda}$ are the operators
	\begin{equation*}
		Q_{\lambda}:\cD(Q_{\lambda})\subset H^1([-\tau_0,\tau_0],\R^d)\to L^{2}([-\tau_0,\tau_0],\R^d),\quad [Q_{\lambda}u](t)=\dot{u}(t)
	\end{equation*}
defined on the domains
	\begin{align*}\label{Qdomains}
		\cD(Q_{\lambda})=\set{u\in H^1([-\tau_0,\tau_0],\R^d)\mid u(-\tau_0)\in N(P_\lambda^-(0)),\, u(\tau_0)\in R(P_\lambda^+(0))}.
	\end{align*}
Let us firstly note that the operators $Q_\lambda$ are Fredholm of index $0$, which follows as for $L^0_\lambda$ in the proof of Lemma \ref{Lemma-FredInd0}. Moreover, the kernel $N(Q_\lambda)$ is isomorphic to $N(P_\lambda^-(0))\cap R(P_\lambda^+(0))=E^u_\lambda(0)\cap E^s_\lambda(0)$ and thus $Q_\lambda$ is an isomorphism for all $\lambda\in\Lambda_0$ by $(A2)$ and Theorem \ref{thm-invertible}. Consequently, $\ind(Q)\in \KO(\Lambda,\Lambda_0)$ is defined.\\
We now simplify notation by $\Phi_\lambda(t):=\Phi_\lambda(t,0)$ and consider the multiplication operators
	\begin{align*}
		M_{\lambda}&\in GL(L^2([-\tau_0,\tau_0],\R^d)),&
		[M_{\lambda} u](t)=\Phi_\lambda(t)u(t)\fall t\in [-\tau_0,\tau_0].
	\end{align*}
Then $M_{\lambda}^{-1}\circ L_{\lambda}^0\circ M_{\lambda}$ are Fredholm operators of index $0$ on the domains
	\begin{align*}
&\cD(M_{\lambda}^{-1}\circ L_{\lambda}^0 \circ M_{\lambda})=\{M_{\lambda}^{-1}u\in H^1([-\tau_0,\tau_0],\R^d): u\in \cD(L_{\lambda}^0)\,\}\\
&=\{M_{\lambda}^{-1}u\in H^1([-\tau_0,\tau_0],\R^d) : u(-\tau_0)\in N(P_\lambda^-(-\tau_0)),\, u(\tau_0)\in R(P_\lambda^+(\tau_0))\,\}\\
&=\{v\in H^1([-\tau_0,\tau_0],\R^d) : (M_{\lambda}v)(-\tau_0)\in N(P_\lambda^-(-\tau_0)),\, (M_{\lambda}v)(\tau_0)\in R(P_\lambda^+(\tau_0))\,\}\\
&=\set{v\in H^1([-\tau_0,\tau_0],\R^d): v(-\tau_0)\in M_{\lambda}^{-1}(-\tau_0)N(P_\lambda^-(-\tau_0)),\, v(\tau_0)\in M_{\lambda}^{-1}(\tau_0)R(P_\lambda^+(\tau_0))}\\
&=\set{v\in H^1([-\tau_0,\tau_0],\R^d):\, v(-\tau_0)\in N(P_\lambda^-(0)), v(\tau_0)\in R(P_\lambda^+(0))},
	\end{align*}
where we have used in the last equality that by \eqref{invariance}
\begin{equation*}
\Phi_\lambda(t)P_\lambda^\pm(0)\Phi_{\lambda}(t)^{-1}=P_{\lambda}^{\pm}(t) \text{ for all } t\in\R_\pm.
\end{equation*}
 Moreover, for $u\in\cD(M_{\lambda}^{-1}\circ  L_{\lambda}^0 \circ M_{\lambda})$ we have the identity
\begin{align*}
		[\left(M_{\lambda}^{-1}\circ  L_{\lambda}^0 \circ M_{\lambda}\right) u](t)
		&=
		\Phi_\lambda(t)^{-1}(\dot{\Phi}_\lambda(t)u(t)+\Phi_\lambda(t)\dot{u}(t)-A_\lambda(t)\Phi_\lambda(t)u(t))\\
		&=
		\dot{u}(t)+\Phi_\lambda(t)^{-1}(\dot{\Phi}_\lambda(t)-A_\lambda(t)\Phi_\lambda(t))u(t)
		=
		\dot{u}(t)
\end{align*}
a.e.\ on $\R$, which shows that $M_{\lambda}^{-1}\circ L_{\lambda}^0 \circ M_{\lambda}=Q_{\lambda}$. Thus, by $(i)$ and $(iii)$ in Proposition \ref{index-bundle},
\begin{align*}
		\ind(Q)=\ind(M^{-1}\diamond L_P \diamond M)=\ind(M^{-1})+\ind(L_P)+\ind(M)=\ind(L_P).
\end{align*}

\subsubsection*{Step IV: $\ind(Q)=[E^u(0)\oplus E^s(0),\Theta(\mathbb{R}^d),\mathcal{L}]\in \KO(\Lambda,\Lambda_0)$}
In this final step of the proof we show that
\begin{equation}\label{Step-4}
\ind(Q)=[E^u(0)\oplus E^s(0),\Theta(\mathbb{R}^d),\mathcal{L}],
\end{equation}
where the bundle homomorphism $\mathcal{L}:E^u(0)\oplus E^s(0)\rightarrow\Theta(\mathbb{R}^d)$ is defined by $\mathcal{L}(\lambda,u,v)=(\lambda,u-v)$.\\
Let us first recall the well-known fact that $L^{2}([-\tau_0,\tau_0],\R^d)=Y_0\oplus Y_1$, where $Y_0$ is the $d$-dimensional space of constant $\R^d$-valued functions and
\begin{align*}
	Y_1=\left\{u\in L^{2}([-\tau_0,\tau_0],\R^d) : \int^{\tau_0}_{-\tau_0}u(t)\d t=0\,\right\}.
\end{align*}
If we now set for $u\in Y_1$
\[
v(t)=\int^{t}_{-\tau_0}u(s)\d s,
\]
then $v\in H^{1}([-\tau_0,\tau_0],\R^d)$, $v(-\tau_0)=0\in N(P_\lambda^-(0))$ and $v(\tau_0)=0\in R(P_\lambda^+(0))$, which shows that $v\in\cD(Q_{\lambda})$. As $Q_\lambda v=u$, it follows that $Y_1\subset R(Q_\lambda)$ and thus $Y_0$ is transversal to the image of $Q$, i.e.,
\begin{equation*}
R(Q_{\lambda})+Y_0=L^{2}([-\tau_0,\tau_0],\R^d).
\end{equation*}
Consequently, the vector bundle $E(Q,Y_0)$ is defined and its fibres are given by
\begin{align*}
E(Q,Y_0)_\lambda&=Q_{\lambda}^{-1}Y_0=\{u\in\cD(Q_{\lambda}): \dot{u}\equiv\text{constant}\}\\
&=\left\{u_a^b(t):=\tfrac{1}{2}\left(1+\tfrac{t}{\tau_0}\right)b+\tfrac{1}{2}\left(1-\tfrac{t}{\tau_0}\right)a: a\in N(P_\lambda^-(0)),\,b\in R(P_\lambda^+(0))\right\}.
\end{align*}
Moreover, $Q_{\lambda}$ acts on the fibres $E(Q,Y_0)_\lambda$ into $Y_0$ by
\begin{equation*}
 Q_{\lambda}\left(u_a^b\right)=\frac{1}{2\tau_0}(b-a).
 \end{equation*}
Finally, we consider the commutative diagram
\begin{align*}
\xymatrix{
E(Q,Y_0)_{\lambda}\ar[r]^{Q_{\lambda}}\ar[d]_{e_{\lambda}}^{\cong}& Y_0\ar[d]^m_{\cong}\\
N(P_\lambda^-(0))\oplus R(P_\lambda^+(0))\ar[r]^-{\mathcal{L}_{\lambda}}&\R^d,
}
\end{align*}
where the homomorphisms $e_{\lambda}$, $\mathcal{L}_\lambda$ and $m$ are defined by
\begin{align*}
& e_{\lambda}(u):=(u(-\tau_0),u(\tau_0)),\quad m(u):=2\tau_0 u,\quad \mathcal{L}_{\lambda}(u,v)=v-u.
\end{align*}
As $N(P_\lambda^-(0))=E^u_\lambda(0)$ and $R(P_\lambda^+(0))= E^s_\lambda(0)$, $\lambda\in\Lambda$, this shows \eqref{Step-4} and finally completes the proof of Theorem \ref{thm-main-lin}.


\section{An Example}\label{section-examples}
In this section, we consider the systems
\begin{equation}\label{lin-HamiltonII}
\left\{
\begin{array}{l}
\dot{u}(t)= A_\lambda(t)u(t), \\
\lim\limits_{t\to\pm\infty} u(t)=0
\end{array}
\right.
\end{equation}
where, for all $\lambda$ in some compact topological space $\Lambda$,
\[
A_\lambda(t)=
\begin{cases}
a(t)\, B_\lambda & \text{for } t\le 0,\\
a(t)\, C_\lambda & \text{for } t\ge 0,
\end{cases}
\]
and
\begin{itemize}
\item[(i)] $a:\R\rightarrow(-\infty,0]$ is continuous, $a(0)=0$, and there are $a^\pm>0$ and $t_0>0$ such that
\[-a^-\leq a(t)\leq -a^+,\quad |t|\geq t_0.\]
\item[(ii)] $\{B_\lambda\}_{\lambda\in\Lambda}$ and $\{C_\lambda\}_{\lambda\in\Lambda}$ are continuous families of symmetric matrices in $\GL(d,\mathbb{R})$ for some $d\geq 2$.
\item[(iii)] for some $1\leq k\leq d-1$, the matrices $B_\lambda$, $C_\lambda$ have $k$ positive and $d-k$ negative eigenvalues, where multiplicities are taken into account.
\end{itemize}
The families of transition matrices $\{\Phi^-_\lambda\}_{\lambda\in\Lambda}$ on $(-\infty,0]$ and $\{\Phi^+_\lambda\}_{\lambda\in\Lambda}$ on $[0,\infty)$ are given by
\[\Phi^-_\lambda(t,s)=\exp\left(\int^t_s a(r)\, dr\, B_\lambda\right)\quad\text{and}\quad \Phi^+_\lambda(t,s)=\exp\left(\int^t_s a(r)\, dr\, C_\lambda\right).\]
We set for $\lambda\in\Lambda$
\begin{equation}
P^-_\lambda=\chi_{(0,\infty)}(B_\lambda),\quad P^+_\lambda=\chi_{(0,\infty)}(C_\lambda),
\end{equation}
where we use functional calculus and $\chi_{(0,\infty)}$ is a characteristic function. Note that the families $\{P^\pm_\lambda\}_{\lambda\in\Lambda}$ depend continuously on the parameter $\lambda$ as the matrices $B_\lambda$, $C_\lambda$ are invertible by assumption. Let now $\lambda\in\Lambda$ be fixed and let
\begin{align*}
\mu_1(B_\lambda)\geq\cdots\geq\mu_{k}(B_\lambda)&>0>\mu_{k+1}(B_\lambda)\geq\cdots\geq\mu_d(B_\lambda)\\
\mu_1(C_\lambda)\geq\cdots\geq\mu_{k}(C_\lambda)&>0>\mu_{k+1}(C_\lambda)\geq\cdots\geq\mu_d(C_\lambda)
\end{align*}
be the eigenvalues of $B_\lambda$, $C_\lambda$, as well as $\{b_{1,\lambda},\ldots,b_{d,\lambda}\}$, $\{c_{1,\lambda},\ldots,c_{d,\lambda}\}$ orthonormal bases of $\R^d$ such that $B_\lambda b_{l,\lambda}=\mu_l(B_\lambda) b_{l,\lambda}$ and $C_\lambda c_{l,\lambda}=\mu_l(C_\lambda) c_{l,\lambda}$ for $1\leq l\leq d$. Then we have for $x\in\mathbb{R}^d$
\begin{align*}
P^-_\lambda x=\sum^k_{i=1}\langle x,b_{i,\lambda}\rangle b_{i,\lambda},\qquad P^+_\lambda x=\sum^k_{i=1}\langle x,c_{i,\lambda}\rangle c_{i,\lambda}
\end{align*}
and obtain for $t,s\in(-\infty,0]$
\[
P^-_\lambda(\Phi^-_\lambda(t,s)x)=\Phi^-_\lambda(t,s)P^-_\lambda x,\quad x\in\R^d,
\]
as well as for $t,s\in[0,\infty)$
\begin{align*}
P^+_\lambda(\Phi^+_\lambda(t,s)x)&=\Phi^+_\lambda(t,s)P^+_\lambda x,\quad x\in\R^d.
\end{align*}
Thus $P^-_\lambda$ is an invariant projector for $\dot{x}=A_\lambda(t)x$ on $(-\infty,0]$, and $P^+_\lambda$ on $[0,\infty)$. If we use that $(I_d-P^-_\lambda)x=\sum^{d}_{i=k+1}\langle x,b_{i,\lambda}\rangle b_{i,\lambda}$, it is readily seen that
\begin{equation*}
\|\Phi^-_\lambda(t,s)P^-_\lambda\|\leq e^{\mu(B_\lambda) a^+(s-t)}\quad\text{and}\quad \|\Phi^-_\lambda(s,t)(I_d-P^-_\lambda)\|\leq e^{\mu(B_\lambda) a^+(s-t)}
\end{equation*}
for all $s\leq t$ with $t,s\in (-\infty,-t_0]$, where $\mu(B_\lambda)=\min_{i=1,\ldots,d}|\mu_i(B_\lambda)|>0$ and we have used the upper bound in $(i)$. Similarly,
\begin{equation*}
\|\Phi^+_\lambda(t,s)P^+_\lambda\|\leq e^{\mu(C_\lambda) a^+(s-t)}\quad\text{and}\quad \|\Phi^+_\lambda(s,t)(I_d-P^+_\lambda)\|\leq e^{\mu(C_\lambda) a^+(s-t)}
\end{equation*}
for all $s\leq t$ with $t,s\in [t_0,\infty)$, where $\mu(C_\lambda)=\min_{i=1,\ldots,d}|\mu_i(C_\lambda)|>0$. Note that, as recalled in Section \ref{section-expdich}, it already follows from these estimates on $(-\infty,-t_0]$ and $[t_0,\infty)$ that $(A1)$ holds, i.e., $\dot{x}=A_\lambda(t) x$ has an exponential dichotomy on both half-axes $\R^-_0$ and $\R^+_0$. For later reference, we note that the reader can easily check that actually
\begin{equation*}
\|\Phi^-_\lambda(t,s)P^-_\lambda\|\leq K_B e^{\alpha_B(s-t)},\quad \|\Phi^-_\lambda(s,t)(I_d-P^-_\lambda)\|\leq K_B e^{\alpha_B(s-t)},\quad t,s\in(-\infty,0],\, t\geq s,
\end{equation*}
for $K_B=e^{\mu(B_\lambda) a^+t_0}$ and $\alpha_B=\mu(B_\lambda) a^+$, and corresponding inequalities hold on $[0,\infty)$ for $\Phi^+_\lambda$ and $K_C=e^{\mu(C_\lambda) a^+t_0}$, $\alpha_C=\mu(C_\lambda) a^+$.\\ 
The stable and unstable subspaces are
\begin{align*}
E^u_\lambda(0)=N(P^-_\lambda)=R(\chi_{(-\infty,0)}(B_\lambda)),\qquad
E^s_\lambda(0)=R(P^+_\lambda)=R(\chi_{(0,\infty)}(C_\lambda))
\end{align*}
and, as $\dim(E^u_\lambda(0))+\dim(E^s_\lambda(0))=d-k+k=d$, it follows from Theorem \ref{thm-invertible} that there is an exponential dichotomy on all of $\mathbb{R}$ for $\dot{x}=A_\lambda(t) x$ if and only if
\begin{align}\label{transversality}
E^u_\lambda(0)\cap E^s_\lambda(0)=\{0\}.
\end{align}
Thus, if $\Lambda_0\neq\emptyset$ is a closed subset of $\Lambda$ such that \eqref{transversality} holds for all $\lambda\in\Lambda_0$, then the index bundle $\ind(L)$ of the corresponding family \eqref{L} is defined and given by
\begin{align}\label{indexam}
\ind(L)=[E^-(B)\oplus E^+(C),\Theta(\mathbb{R}^d),\mathcal{L}]\in \KO(\Lambda,\Lambda_0),
\end{align}
where $E^-(B)\oplus E^+(C)$ is the vector bundle over $\Lambda$ induced by the family of projections
\[
\{\chi_{(-\infty,0)}(B_\lambda))\oplus \chi_{(0,\infty)}(C_\lambda))\}_{\lambda\in\Lambda}
\]
in $\mathbb{R}^{2d}$. Note that the fibres
\begin{align*}
E^-(B)_\lambda=\bigoplus_{\mu<0}N(\mu I_d-B_\lambda),\qquad
E^+(C)_\lambda=\bigoplus_{\mu>0}N(\mu I_d-C_\lambda)
\end{align*}
are made by the eigenspaces of the symmetric matrices $B_\lambda$ and $C_\lambda$.\\
Let us now consider the special case that $E^u_\lambda(0)\perp E^s_\lambda(0)$ for all $\lambda\in\Lambda_0$. Then $N(P^-_\lambda)\perp R(P^+_\lambda)$ which shows that $P^-_\lambda=P^+_\lambda$ for $\lambda\in\Lambda_0$. Thus the invariant projector $P_\lambda:=P^-_\lambda=P^+_\lambda$ yields an exponential dichotomy for $\dot{x}=A_\lambda(t) x$ on all of $\mathbb{R}$, where $K:=\max\{K_B,K_C\}$ and $\alpha:=\min\{\alpha_B,\alpha_C\}$. Then, if $Q:\Lambda\times\mathbb{R}\rightarrow\Mat(d,\mathbb{R})$ is continuous and
\begin{align}\label{Qbound}
\sup_{(\lambda,t)\in\Lambda\times\R}|Q_\lambda(t)|<\frac{\alpha}{4 K^2},
\end{align}
the equations
\begin{align}\label{perturbed}
\dot{x}=(A_\lambda(t)+Q_\lambda(t))x
\end{align}
satisfy $(A1)$ and $(A2)$ by Theorem \ref{roughness}. As the same is true for the matrix family $A_\lambda(t)+s\,Q_\lambda(t)$, $s\in[0,1]$, we see from the homotopy invariance in Proposition \ref{index-bundle} that the index bundle of the corresponding operators \eqref{L} for the perturbed problem \eqref{perturbed} is given by \eqref{indexam} as well.\\
The final aim of this section is to consider numerical examples of the above equations for $d=2$ and one-dimensional parameter spaces. We begin by $\Lambda=S^1$, $\Lambda_0=\{-1\}$ and
\begin{align}\label{ex-BC}
B_\lambda=\begin{pmatrix}
-1&0\\
0&1
\end{pmatrix},\qquad C_\lambda=\begin{pmatrix}
\cos\theta&\sin\theta\\
\sin\theta&-\cos\theta
\end{pmatrix},\,\lambda=e^{i\theta},\,0\leq\theta\leq 2\pi.
\end{align}
Note that for $\lambda_0=-1\in S^1$, we now indeed have that $P^-_{\lambda_0}=P^+_{\lambda_0}$ is an exponential dichotomy on all of $\mathbb{R}$, which shows $(A2)$ for $\Lambda_0=\{-1\}$. Moreover, $\mu_1(B_\lambda)=\mu_1(C_\lambda)=1$, $\mu_2(B_\lambda)=\mu_2(C_\lambda)=-1$, and
\[
E^-(B)_\lambda=\spann\left\{\begin{pmatrix}
1\\
0
\end{pmatrix}\right\},\quad E^+(C)_\lambda=\spann{\renewcommand{\arraystretch}{1.3}\left\{\begin{pmatrix}
\cos\frac{\theta}{2}\\\sin\frac{\theta}{2}
\end{pmatrix}\right\}}.
\]
If we now apply the map \eqref{alpha} to \eqref{indexam}, we obtain for the corresponding operators \eqref{L}
\begin{align*}
\Psi_{\lambda_0}(\ind(L))=[E^-(B)\oplus E^+(C)]-[\Theta(\R^2)]=[E^+(C)]-[\Theta(\R)]\neq 0\in\widetilde{\KO}(S^1)\cong\mathbb{Z}_2,
\end{align*}
as $E^+(C)$ is the M\"obius bundle over $S^1$. Consequently, $\Psi_{\lambda_0}(\ind(L))$ generates $\widetilde{\KO}(S^1)$ and in particular is non-trivial. That the index bundle $\Psi_{\lambda_0}(\ind(L))$ is non-trivial for this choice of $B_\lambda$ and $C_\lambda$ was already obtained by Pejsachowicz in \cite[\S 4]{JacoboAMS} under the additional assumption $\lim_{t\rightarrow-\infty}a(t)=\lim_{t\rightarrow\infty}a(t)=-1$, which makes $A_\lambda(t)$ asymptotically hyperbolic as in Corollary \ref{cor-Jacobo}. Here we can go even further beyond this assumption by considering perturbations \eqref{perturbed}, where $Q$ satisfies \eqref{Qbound} for the constants $\alpha$ and $K$ that depend on the numbers $t_0$ and $a^+$ in $(i)$. \\
Next, let us consider for $\Lambda=[0,\pi]$ and $\Lambda_0=\{0,\pi\}$ the matrices
\begin{align}\label{BC}
B_\theta=\begin{pmatrix}
\cos\theta&-\sin\theta\\
-\sin\theta&-\cos\theta
\end{pmatrix}\quad\text{and}\quad C_\theta=\begin{pmatrix}
\cos\theta&\sin\theta\\
\sin\theta&-\cos\theta
\end{pmatrix},\,\,0\leq\theta\leq \pi.
\end{align}
Then $P^-_{\theta}=P^+_{\theta}$ for $\theta=0$ and $\theta=\pi$, which in particular shows that $(A2)$ holds for $\Lambda_0=\{0,\pi\}$. Moreover, as above, $\mu_1(B_\theta)=\mu_1(C_\theta)=1$, $\mu_2(B_\theta)=\mu_2(C_\theta)=-1$, but now
\begin{align}\label{basis}
E^-(B)_\theta=\spann{\renewcommand{\arraystretch}{1.3}\left\{\begin{pmatrix}
\sin\frac{\theta}{2}\\
\cos\frac{\theta}{2}
\end{pmatrix}\right\}},\quad E^+(C)_\theta=\spann{\renewcommand{\arraystretch}{1.3}\left\{\begin{pmatrix}
\cos\frac{\theta}{2}\\\sin\frac{\theta}{2}
\end{pmatrix}\right\}}.
\end{align}
Let us emphasize that, even though we can easily extend $\{B_\theta\}_{\theta\in[0,\pi]}$ and $\{C_\theta\}_{\theta\in[0,\pi]}$ to families parametrised by $S^1$, Theorem \ref{thm-main-lin} and \eqref{alpha} only yield in this case
\begin{align*}
\Psi_{\lambda_0}(\ind(L))&=[E^-(B)\oplus E^+(C)]-[\Theta(\R^2)]\\
&=([E^-(B)]-[\Theta(\R)])+([E^+(C)]-[\Theta(\R)])=0\in\widetilde{\KO}(S^1)\cong\mathbb{Z}_2,
\end{align*}
as $E^+(C)$ and $E^-(B)$ are both isomorphic to the M\"obius bundle over $S^1$. On the other hand, if we consider as before $\Lambda=[0,\pi]$, $\Lambda_0=\{0,\pi\}$, then $(A3)$ holds as $\Lambda$ is contractible, and we obtain from Corollary \ref{cor-main-triv}
\[\ind(L)=[\Theta(\R^2),\Theta(\R^2),\mathcal{M}],\]
where
\[\mathcal{M}_\theta={\renewcommand{\arraystretch}{1.3}\begin{pmatrix}
\sin\frac{\theta}{2}&\cos\frac{\theta}{2}\\
\cos\frac{\theta}{2}&\sin\frac{\theta}{2}
\end{pmatrix}},
\]
and $\mathcal{L}_D(\theta)=\sin^2\frac{\theta}{2}-\cos^2\frac{\theta}{2}=-\cos(\theta)$, $\theta\in[0,\pi]$, in Corollary \ref{cor-main-contractible}. Consequently, it follows from the latter corollary that $\ind(L)\in\KO([0,\pi],\{0,\pi\})$ is non-trivial, or in other words
\begin{equation}\label{beta-non-trivial}
\psi_{0,\pi}(\ind(L))=1\in\mathbb{Z}_2
\end{equation}
under the isomorphism $\psi_{0,\pi}$ in \eqref{beta}.\\
Note that the matrices $A_\lambda(t)$ in \eqref{lin-HamiltonII} are asymptotically hyperbolic under our assumptions if the limits $\lim_{t\rightarrow\pm\infty} A_\lambda(t)$ exist, in which case the above yields a non-trivial example of Hu and Portaluri's Theorem 1 in \cite{HuP}. Here we can again not only lift the existence of the limits, but even consider perturbations \eqref{perturbed}, where $Q$ satisfies \eqref{Qbound} for the constants $\alpha$ and $K$ that depend on the numbers $t_0$ and $a^+$ in $(i)$.

\section{Bifurcation of Homoclinic Orbits}\label{section-bifurcation}

\subsection{The Bifurcation Theorem and Corollaries}
In this final part of our work, we give a first application of Theorem \ref{thm-main-lin} to multiparameter bifurcation of homoclinic orbits, where thus we consider
\begin{equation}\label{Nonlinear}
\left\{
\begin{array}{l}
\dot{u}(t)= g(\lambda,t,u(t)), \\
\lim\limits_{t\to\pm\infty} u(t)=0,
\end{array}
\right.
\end{equation}
for a continuous map $g:\Lambda\times\mathbb{R}\times\mathbb{R}^d \rightarrow\mathbb{R}^d$ such that the partial derivative $D_ug$ exists and is continuous on $\Lambda\times\mathbb{R}\times\mathbb{R}^d$ as well. Moreover, we assume that $g$ and $D_ug$ are bounded, as well as $g(\lambda,t,0)=0$ for all $(\lambda,t)\in \Lambda\times\mathbb{R}$. In particular, $u\equiv0$ is a solution of \eqref{Nonlinear} for all $\lambda\in\Lambda$.\\
We call $\lambda^\ast\in\Lambda$ a bifurcation point if in every neighbourhood of $(\lambda^\ast,0)\in \Lambda\times H^1(\mathbb{R},\mathbb{R}^d)$, there is some $(\lambda,u)$ such that $u$ is a solution of \eqref{Nonlinear} and $u\neq 0$. Henceforth, $\mathcal{B}\subset\Lambda$ denotes the set of all bifurcation points of \eqref{Nonlinear}. The linearisations of \eqref{Nonlinear} are
\begin{equation}\label{Linear}
\left\{
\begin{array}{l}
\dot{u}(t)= A_\lambda(t)u(t), \\
\lim\limits_{t\to\pm\infty} u(t)=0,
\end{array}
\right.
\end{equation}
where
\begin{align}\label{Alambda}
A_\lambda(t):=D_ug(\lambda,t,0)\in\Mat(d,\mathbb{R}).
\end{align}
In what follows we let $\Sigma\subset\Lambda$ denote those $\lambda\in\Lambda$ for which \eqref{Linear} has other solutions than the trivial one $u\equiv 0$. Note that $(A2)$ implies that $\Lambda_0\cap\Sigma=\emptyset$.\\
Our main theorem on bifurcation for \eqref{Nonlinear} is as follows.
\begin{theorem}\label{main-bif}
Assume that $(A1)$, $(A2)$ and $(A3)$ hold for the equations \eqref{Linear} parameterised by a simply connected space $\Lambda$. If there are $\lambda_0,\lambda_1\in\Lambda_0$ such that
\[\mathcal{L}_D(\lambda_0)\cdot\mathcal{L}_D(\lambda_1)<0,\]
then the bifurcation set $\mathcal{B}$ disconnects $\Lambda$.
\end{theorem}
Let us first emphasize the following topological implications about the bifurcation set $\mathcal{B}$ under additional assumptions on $\Lambda$. 
\begin{corollary}\label{cor-bifurcation-topology}
Let the assumptions of Theorem \ref{main-bif} hold and assume in addition that $\Lambda$ is a compact connected manifold of dimension $n\geq 2$ with (possibly empty) boundary $\partial \Lambda$. Then the (Lebesgue covering-) dimension of $\mathcal{B}$ is at least $n-1$, and if $\partial \Lambda\cap\Sigma=\emptyset$, $\mathcal{B}$ is not contractible as a topological space.
\end{corollary}
\begin{proof}
As $\Lambda$ is connected by assumption, we obtain from the long exact sequence in homology
\[\ldots\rightarrow H_1(\Lambda,\Lambda\setminus\mathcal{B};\mathbb{Z}_2)\rightarrow \tilde{H}_0(\Lambda\setminus\mathcal{B};\mathbb{Z}_2)\rightarrow\tilde{H}_0(\Lambda;\mathbb{Z}_2)=0\]
that $H_1(\Lambda,\Lambda\setminus\mathcal{B};\mathbb{Z}_2)\rightarrow \tilde{H}_0(\Lambda\setminus\mathcal{B};\mathbb{Z}_2)$ is surjective. Now $\Lambda\setminus\mathcal{B}$ is not connected by Theorem \ref{main-bif} and thus the reduced homology group $\tilde{H}_0(\Lambda\setminus\mathcal{B};\mathbb{Z}_2)$ is non-trivial. Consequently, $H_1(\Lambda,\Lambda\setminus\mathcal{B};\mathbb{Z}_2)$ is non-trivial as well. Finally, Poincar\'{e}-Lefschetz duality yields an isomorphism 
\[H_1(\Lambda,\Lambda\setminus\mathcal{B};\mathbb{Z}_2)\cong \check{H}^{n-1}(\mathcal{B},\mathcal{B}\cap\partial\Lambda;\mathbb{Z}_2),\]
where $\check{H}$ stands for \v{C}ech-cohomology and we use that $\mathcal{B}$ is compact. This shows the assertion on the dimension, as a non-trivial \v{C}ech-cohomology in degree $n-1$ gives $n-1$ as lower bound on the dimension of the total space $\mathcal{B}$. If in addition $\Sigma\cap\partial\Lambda=\emptyset$, then $\mathcal{B}\cap\partial\Lambda=\emptyset$ as $\mathcal{B}\subset\Sigma$. Hence $\check{H}^{n-1}(\mathcal{B};\mathbb{Z}_2)=\check{H}^{n-1}(\mathcal{B},\mathcal{B}\cap\partial\Lambda;\mathbb{Z}_2)$ which shows that $\mathcal{B}$ is not contractible as $n\geq 2$.  
\end{proof}
The case $\Lambda=[0,1]$ and $\Lambda_0=\{0,1\}$ is worth to be noted separately as it is the main outcome of Hu and Portaluri's work \cite{HuP} under the stronger assumptions $(A4)$ and $(A5)$.
\begin{corollary}\label{cor-bifurcation-III}
Assume that $(A1)$ and $(A2)$ hold for the equations \eqref{Nonlinear} where $\Lambda=[0,1]$ and $\Lambda_0=\{0,1\}$. Then there is a bifurcation point $\lambda^\ast\in(0,1)$ of \eqref{Nonlinear} if 
\[\mathcal{L}_D(0)\cdot\mathcal{L}_D(1)<0.\]
\end{corollary}
Let us emphasize once again that $(A3)$ holds on every contractible space $\Lambda$ and thus it is not an assumption in the previous corollary.

\subsection{Proof of Theorem \ref{main-bif}}
We consider the nonlinear operator family
\[
G:\Lambda\times H^1(\mathbb{R},\mathbb{R}^d)\rightarrow L^2(\mathbb{R},\mathbb{R}^d),\quad G(\lambda,u)=\dot{u}(t)-g(\lambda,t,u(t))
\]
and the corresponding linearisations at the trivial branch of solutions
\[
L:=D_0G:\Lambda\times H^1(\mathbb{R},\mathbb{R}^d)\rightarrow L^2(\mathbb{R},\mathbb{R}^d), \quad L_\lambda u=\dot{u}(t)-A_\lambda(t)u(t),\]
which are of the form \eqref{L} as in our main Theorem \ref{thm-main-lin}. Consequently, the operators $L_\lambda$, $\lambda\in\Lambda$, are Fredholm of index $0$ by Theorem \ref{thm-FredholmIndex}.\\
Let us now briefly recall the definition of the parity for paths in the space $\Phi_0(X,Y)$ of Fredholm operators of index $0$ between Banach spaces $X,Y$, where we mostly follow \cite{Memoirs}. Let $L:[a,b]\rightarrow\Phi_0(X,Y)$ be a path such that $L_a$ and $L_b$ are invertible. Fitzpatrick and Pejsachowicz showed in \cite{FiPejsachowiczI} that there is a parametrix for $L$, i.e., a path $M:[a,b]\rightarrow\GL(Y,X)$ such that $M_\lambda L_\lambda=I_X+K_\lambda$ for some path $K:[a,b]\rightarrow\mathcal{K}(X)$ of compact operators. As $I_X+K_\lambda$ are invertible for $\lambda=a,b$, the Leray-Schauder degree of these operators is defined and given by
\begin{align}\label{LSdeg}
\deg_{LS}(I_X+K_\lambda)=(-1)^{k(\lambda)},
\end{align}
where $k(\lambda)$ denotes the algebraic multiplicity of the number of eigenvalues less than $-1$ of $K_\lambda$. The parity $\sigma(L,[a,b])$ is defined as the unique element in $\mathbb{Z}_2=\{0,1\}$ such that
\[\deg_{LS}(I_E+K_a)\deg_{LS}(I_E+K_b)=(-1)^{\sigma(L,[a,b])}.\]
It was shown in \cite{FiPejsachowiczI}, \cite{Memoirs} that this definition indeed does not depend on the choice of the parametrix $M$. Moreover, the parity has the following properties:
\begin{itemize}
 \item[(i)] If $L_\lambda\in\GL(X,Y)$ for all $\lambda\in[a,b]$, then $\sigma(L,[a,b])=0$.
 \item[(ii)] If $h:[a,b]\times[0,1]\rightarrow\Phi_0(X,Y)$ is a homotopy such that $h(a,s), h(b,s)\in\GL(X,Y)$ for all $s\in[0,1]$, then
 \[\sigma(h(\cdot,0),[a,b])=\sigma(h(\cdot,1),[a,b]).\]
 \item[(iii)] If $E=E_1\oplus E_2$ for two closed subspaces of $E$ such that $L_\lambda(E_i)\subset E_i$ for all $\lambda\in[a,b]$, $i=1,2$, then
 \[\sigma(L,[a,b])=\sigma(L\mid_{E_1},[a,b])+\sigma(L\mid_{E_2},[a,b]).\]
 \end{itemize}
The main motivation for introducing the parity is bifurcation theory of nonlinear operator equations as in the following theorem that can be found in \cite{FiPejsachowiczProc}.
\begin{theorem}\label{Rabier}
Let $X,Y$ be real Banach spaces and $G:[a,b]\times X\rightarrow Y$ a $C^1$-map with $G(\lambda,0)=0$ for all $\lambda\in[a,b]$. Suppose that the derivatives $D_uG_\lambda$ of $G_\lambda:X\rightarrow Y$ at $0\in X$ are Fredholm of index $0$, and that $D_uG_{\lambda}\in\GL(X,Y)$ for $\lambda=a,b$ as well as $\sigma(D_uG_\cdot,[a,b])=1$. Then there is a bifurcation point from the trivial branch, i.e., there is some $\lambda^\ast\in(a,b)$ such that in every neighbourhood of $(\lambda^\ast,0)\in [a,b]\times X$, there is some $(\lambda,u)$ such that $G(\lambda,u)=0$ and $u\neq 0$.
\end{theorem}
Fitzpatrick and Pejsachowicz pointed out in \cite{FiPejsachowiczII} that the index bundle \eqref{indPejsachowicz} for $\Lambda=S^1$ and closed paths $L:S^1\rightarrow\Phi_0(X,Y)$ can be identified with the parity under the identification $\widetilde{\KO}(S^1)\cong\mathbb{Z}_2$. This fact has been applied various times to study bifurcation problems of ordinary and partial differential equations, e.g., in \cite{Jacobo, JacoboTMNAI}. Let us finally note that there is a subtle gap in Hu and Portaluri's work \cite{HuP} as the authors call their invariant $\sigma(L_\lambda,\lambda\in [a,b])$ the parity, but nowhere explain that it actually coincides with the number $\sigma(L,[a,b])$ used in Theorem \ref{Rabier}.\\
Fitzpatrick and Pejsachowicz proved in \cite{Memoirs} the following theorem.
\begin{theorem}\label{thm-parityindbundle}
Let $L:([a,b],\{a,b\})\rightarrow(\Phi_0(X,Y),\GL(X,Y))$ be a path and $V\subset Y$ a finite dimensional subspace that is transversal to the range of $L$ as in \eqref{subspace}. If $\psi:[a,b]\times\mathbb{R}^d\rightarrow E(L,V)$ is any trivialisation of the bundle $E(L,V)$ in \eqref{E-V-L} over $[a,b]$, then
\begin{align*}
\sigma(L,[a,b])=\sgn\det(L_a\psi_a)\sgn\det(L_b\psi_b)\in\mathbb{Z}_2.
\end{align*}
\end{theorem}
Now let $\lambda_0,\lambda$ be as in the statement of Theorem \ref{main-bif} and let $\gamma:[0,1]\rightarrow\Lambda$ be a path such that $\gamma(0)=\lambda_0$ and $\gamma(1)=\lambda_1$. If we consider the path $L_\gamma=\{L_{\gamma(s)}\}_{s\in[0,1]}$ of Fredholm operators of index $0$, then $\sigma(L_\gamma,[0,1])\in\mathbb{Z}_2$ is defined. Moreover, it follows from Proposition \ref{beta-iso} for $\Lambda=[0,1]$, $\Lambda_0=\{0,1\}$ and Theorem \ref{thm-parityindbundle} that $\sigma(L_\gamma,[0,1])=\psi_{0,1}(\ind(L_\gamma))\in\mathbb{Z}_2$. Now $\ind(L_\gamma)$ is non-trivial if and only if the map $\mathcal{L}^\gamma_D:[0,1]\rightarrow\R$ in Corollary \ref{cor-main-contractible} satisfies $\mathcal{L}^\gamma_D(0)\cdot\mathcal{L}^\gamma_D(1)<0$. Finally, if we consider the map $\mathcal{L}_D:\Lambda\rightarrow\R$ on all of $\Lambda$, then $\mathcal{L}_D(\gamma(t))=\mathcal{L}^\gamma_D(t)$ for all $t\in[0,1]$, and thus by assumption
\[\mathcal{L}^\gamma_D(0)\cdot\mathcal{L}^\gamma_D(1)=\mathcal{L}_D(\lambda_0)\cdot\mathcal{L}_D(\lambda_1)<0.\]    
In summary, $\sigma(L_\gamma,[0,1])=1$ and we obtain from Theorem \ref{Rabier} that there is some $t'\in(0,1)$ such that $\gamma(t')\in\mathcal{B}$, which thus is not the empty set.\\
To show that $\mathcal{B}$ disconnects $\Lambda$, we follow \cite{PortW14} and assume that there is a path $\tilde{\gamma}$ in $\Lambda$ such that $\tilde{\gamma}(0)=\lambda_1$, $\tilde{\gamma}(1)=\lambda_0$ and $\sigma(L_{\tilde{\gamma}},[0,1])=0\in\mathbb{Z}_2$. Then the concatenation of $L_\gamma$ and $L_{\tilde{\gamma}}$ is a closed path and thus homotopic to a constant path as $\Lambda$ is simply connected by assumption. The properties $(i)$ and $(ii)$ of the parity from above imply that the parity of this closed path vanishes. On the other hand, property $(iii)$ shows that the parity of the latter path is non-trivial, which is a contradiction. Thus $\sigma(L_\gamma,[0,1])=1$ for any path $\gamma$ in $\Lambda$ that joins $\lambda_0$ and $\lambda_1$, which implies by Theorem \ref{Rabier} that any such path intersects $\mathcal{B}$. As $\lambda_0$, $\lambda_1\notin\Sigma$ and $(\Lambda\setminus\Sigma)\cap\mathcal{B}=\emptyset$, it follows that $\mathcal{B}$ indeed disconnects $\Lambda$.


\subsection{Examples}
In this section we aim to illustrate Theorem \ref{main-bif}, where firstly we continue the examples that we already considered in Section \ref{section-examples}. Let us begin by assuming that $\Lambda=S^1$, $\Lambda_0=\{-1\}$ and the linearisations of \eqref{Nonlinear} are of the form \eqref{lin-HamiltonII} for the matrices \eqref{ex-BC}. It follows from our findings in Section \ref{section-examples} that \eqref{lin-HamiltonII} has a non-trivial solution only for $\lambda=1\in S^1$. As $\mathcal{B}\subset\Sigma$, this is the only possible bifurcation point of \eqref{Nonlinear}. We leave it to the reader to check by Corollary \ref{cor-bifurcation-III} that actually $\mathcal{B}=\{1\}$. Thus even though $\mathcal{B}$ is non-empty, it does not disconnect $\Lambda=S^1$. Note that this does not contradict Theorem \ref{main-bif} as neither $\Lambda=S^1$ is simply connected nor do maps $s_1,s_2:S^1\rightarrow\mathbb{R}^2$ as in $(A3)$ exist.\\
Let us now consider an example where the assumptions of Theorem \ref{main-bif} hold. We let $\Lambda$ be any compact and simply connected space and consider a continuous map $\theta:\Lambda\rightarrow\mathbb{R}$. We assume that the linearisations of \eqref{Nonlinear} are of the form \eqref{lin-HamiltonII}, where for $\lambda\in\Lambda$
\begin{align*}
B_{\lambda}=\begin{pmatrix}
\cos(\theta(\lambda))&-\sin(\theta(\lambda))\\
-\sin(\theta(\lambda))&-\cos(\theta(\lambda))
\end{pmatrix}\quad\text{and}\quad C_{\lambda}=\begin{pmatrix}
\cos(\theta(\lambda))&\sin(\theta(\lambda))\\
\sin(\theta(\lambda))&-\cos(\theta(\lambda))
\end{pmatrix}.
\end{align*}
Note that $P^-_{\lambda}=P^+_{\lambda}$ if $\theta(\lambda)=k\pi$ for some $k\in\mathbb{Z}$, which in particular shows that $(A2)$ holds for any subset $\Lambda_0$ of $\Lambda$ which is contained in $\theta^{-1}(\{k\pi:\, k\in\mathbb{Z}\})$. Moreover, $\mu_1(B_\lambda)=\mu_1(C_\lambda)=1$, $\mu_2(B_\lambda)=\mu_2(C_\lambda)=-1$, and if we define continuous maps $s_1, s_2:\Lambda\rightarrow\mathbb{R}^2$ by 
\begin{align*}
s_1(\lambda)=\begin{pmatrix}
\sin\frac{\theta(\lambda)}{2}\\
\cos\frac{\theta(\lambda)}{2}
\end{pmatrix},\quad s_2(\lambda)=\begin{pmatrix}
\cos\frac{\theta(\lambda)}{2}\\\sin\frac{\theta(\lambda)}{2}
\end{pmatrix},
\end{align*}
then $E^u(0)_\lambda=\spann\{s_1(\lambda)\}$ and $E^s(0)_\lambda=\spann\{s_2(\lambda)\}$ for all $\lambda\in\Lambda$, which shows $(A3)$, and 
\[\mathcal{L}_D(\lambda)=\det(s_1(\lambda),s_2(\lambda))=\sin^2\frac{\theta(\lambda)}{2}-\cos^2\frac{\theta(\lambda)}{2}=-\cos(\theta(\lambda)).\]
Now it is easy to obtain examples of multiparameter bifurcation for \eqref{Nonlinear} by considering level sets of $\theta:\Lambda\rightarrow\mathbb{R}$. For example, let us consider the closed unit disc $\Lambda=D^2$ and let $\theta:D^2\rightarrow\mathbb{R}$ be such that $\theta|_{S^1}\equiv 0$, $\theta(D^2)\subset[0,\pi]$ and $\theta(0,0)=\pi$. Then we obtain from Corollary \ref{cor-bifurcation-topology} for $\Lambda_0=\{(0,0),(0,1)\}$ that $\mathcal{B}$ disconnects $D^2$, is not contractible as a topological space and of covering dimension at least $1$. Note that if $\theta$ is differentiable and $\frac{\pi}{2}$ is a regular value, then it follows from common Morse theory that $\mathcal{B}$ actually is a submanifold of $D^2$ of dimension $1$. In general, it is an interesting observation that any perturbation as in \eqref{perturbed} does not affect the index bundle as long as it satisfies \eqref{Qbound}. Thus by Theorem \ref{thm-main-lin} also the topological implications of Corollary \ref{cor-bifurcation-topology} still hold and thus the bifurcation set of the perturbed problem still disconnects $D^2$, is not contractible and of covering dimension at least $1$.\\
Another instructive setting appears if we consider maps $\theta:D^2\rightarrow\mathbb{R}$ of the form $\theta(x,y)=\theta_1(x)\theta_2(y)$. For example, if we assume that $\theta_1:[-1,1]\rightarrow[0,\pi]$ is a homeomorphism and $\theta_2(0)=1$, then Corollary \ref{cor-bifurcation-topology} for $\Lambda_0=\{(-1,0), (1,0)\}$ implies that $\mathcal{B}$ disconnects $D^2$, is of covering dimension $1$ and intersects $S^1$ in at least two distinct points. Moreover, there is at least one intersection of $\mathcal{B}$ and $S^1$ in each open semi-circle $S^1_+=\{(x,y)\in S^1:\, y>0\}$ and $S^1_-=\{(x,y)\in S^1:\, y<0\}$. The latter follows as
\[\mathcal{L}_D(-1,0)\cdot\mathcal{L}_D(1,0)<0\]      
and thus $\mathcal{L}_D|_{S^1}:S^1\rightarrow\mathbb{R}$ changes sign along the semi-cicles which implies an intersection with $\mathcal{B}$ by Corollary \ref{cor-bifurcation-III}. Let us once again note that the bifurcation set of any perturbed family \eqref{perturbed} that satisfies \eqref{Qbound} has the same topological properties.\\
The main aim of our upcoming work \cite{MultBif} is to construct a bifurcation invariant for \eqref{Nonlinear} in settings as $\Lambda=D^2$ and $\Lambda_0=S^1$, which consequently is not applicable to both examples on $D^2$ above. In the latter one, this is obviously the case as $\mathcal{B}\cap S^1=\emptyset$ if $\Lambda_0=S^1$, whereas the first example allows an easy deformation to a constant system for which the index bundle vanishes in $KO(D^2,S^1)$. The novelty of the new invariant is that it yields, e.g., the existence of isolated points of $\mathcal{B}\subset D^2$ which cannot be found by the parity and thus are really caused by a multiparameter effect.\\
Finally, we consider an example that underpins the applicability of our findings. In the perturbed second-order equation
\begin{equation}\label{eq:2nd-order-family-semi}
		u''(t)+p(\lambda,t)\,u'(t)+q(\lambda,t)\,u(t)+f(\lambda,t,u(t))=0,
		\qquad \lambda\in[\lambda_0,\lambda_1],
\end{equation}
we assume that $f:[\lambda_0,\lambda_1]\times\mathbb{R}\times\mathbb{R}\to\mathbb{R}$ is continuous in $(\lambda,t,u)$ and continuously differentiable in $u$. Moreover,
\begin{equation}\label{eq:f-assumptions}
		f(\lambda,t,0)=0,\qquad D_u f(\lambda,t,0)=0 \qquad \text{for every } t\in\mathbb{R}\text{ and }\lambda\in [\lambda_0,\lambda_1]
\end{equation}
and the coefficients $p,q:[\lambda_0,\lambda_1]\times\mathbb{R}\to\mathbb{R}$ are continuous in $(\lambda,t)$. Note that \eqref{eq:f-assumptions} implies that $u\equiv 0$ is a solution of \eqref{eq:2nd-order-family-semi} for every $\lambda$. To study bifurcation from this trivial branch of solutions, we consider the family of linear second--order ordinary differential equations 
\begin{equation}\label{eq:2nd-order-family}
		u''(t)+p(t,\lambda)\,u'(t)+q(t,\lambda)\,u(t)=0,
		\qquad \lambda\in[\lambda_0,\lambda_1],
\end{equation}
on $\mathbb{R}$. To shorten the presentation, we do not consider the most general setting in which we could apply Theorem \ref{thm-main-lin}, but assume that the limits
\begin{equation}\label{eq:limits-pq}
		p(\lambda,t)\to p^\pm(\lambda),\qquad q(\lambda,t)\to q^\pm(\lambda)\qquad \text{as }t\to\pm\infty
\end{equation}
exist unformly in $\lambda$ and that the limiting constant--coefficient equations
\[
u''+p^-(\lambda)u'+q^-(\lambda)u=0,
\qquad
u''+p^+(\lambda)u'+q^+(\lambda)u=0
\]
are hyperbolic, i.e.\ their characteristic polynomials
\begin{equation}\label{eq:char-polys}
		r^2+p^-(\lambda)r+q^-(\lambda)=0,
		\qquad
		r^2+p^+(\lambda)r+q^+(\lambda)=0
\end{equation}
have roots $r_i^{\pm}$ such that  $\RE r_1^-\cdot \RE r_2^-<0$ and $\RE r_1^+\cdot \RE r_2^+<0$.\\
The standard phase variables $x_1(t)=u(t)$ and $x_2(t)=u'(t)$, transform the linear equation \eqref{eq:2nd-order-family} into the equivalent non-autonomous linear system
\begin{equation}\label{eq:first-order-system}
	x'(t)=A_\lambda(t)\,x(t),\qquad x(t)=\binom{x_1(t)}{x_2(t)},
\end{equation}
where
\begin{equation*}
	A_\lambda(t)=
	\begin{pmatrix}
		0 & 1\\
		-q(\lambda,t) & -p(\lambda,t)
	\end{pmatrix}.
\end{equation*}
The limits \eqref{eq:limits-pq} imply that $A_\lambda(t)\to A_\lambda^\pm$ as $t\to\pm\infty$, where
\[
A_\lambda^\pm=
\begin{pmatrix}
	0 & 1\\
	-q^\pm(\lambda) & -p^\pm(\lambda)
\end{pmatrix}
\]
and \eqref{eq:char-polys} guarantees that $A_\lambda^\pm$ are hyperbolic matrices as in Assumption $(A4)$.\\
The semilinear equation \eqref{eq:2nd-order-family-semi} transforms to
\begin{equation}\label{eq:first-order-system-semi}
	x'(t)=A_\lambda(t)\,x(t)+F(\lambda,t,x(t)),
\end{equation}
where
\begin{equation}\label{eq:F-def}
	F(\lambda,t,x)=
	\binom{0}{-f(\lambda,t,x_1)}.
\end{equation}
By \eqref{eq:f-assumptions} we have
\[
F(\lambda,t,0)=0,\qquad D_xF(\lambda,t,0)=0\qquad \text{for all }t\in\mathbb{R} \text{ and }\lambda\in [\lambda_0,\lambda_1],
\]
and so the linearised equation of \eqref{eq:first-order-system-semi} at the trivial solution $x\equiv 0$ is exactly \eqref{eq:first-order-system}. Under the above assumptions, there exist one-dimensional spaces of solutions of the linearized equation \eqref{eq:2nd-order-family} spanned by two functions $u_\pm(\lambda,\cdot):\R \rightarrow\R$ such that $u_-(\lambda,t)\to 0$ as $t\to-\infty$ and $u_+(\lambda,t)\to 0$ as $t\to+\infty$. Equivalently, in the first--order formulation \eqref{eq:first-order-system}, the vectors
\begin{equation*}\label{u-v}
v_-^u(\lambda):=\binom{u_-(\lambda,0)}{u_-'(\lambda,0)},
\qquad
v_+^s(\lambda):=\binom{u_+(\lambda,0)}{u_+'(\lambda,0)}
\end{equation*}
span the unstable and stable subspaces at $0\in\mathbb{R}$.\\
If we now consider
\begin{equation}\label{eq:Evans-parity-D}
	\mathcal{L}_D(\lambda)
	:=\det\left(v_-^u(\lambda),\,v_+^s(\lambda)\right)
	=\det\begin{pmatrix}
		u_-(\lambda,0) & u_+(\lambda,0)\\
		u_-'(\lambda,0) & u_+'(\lambda,0)
	\end{pmatrix}
\end{equation}
as in Corollary \ref{cor-bifurcation-III}, then any isolated zero at which $\mathcal{L}_D$ changes its sign is a bifurcation point of \eqref{eq:2nd-order-family-semi}.\\  
For a numerical example of these findings, let us consider the Schr\"{o}dinger-type equation 
\begin{equation*}
	u''(t) + (-\lambda^2 - V(t))u(t) +f(\lambda,t,u) = 0 \qquad \lambda\in[1/2,3/2],
\end{equation*}
for the P\"{o}schl-Teller potential $V(t) = -2\operatorname{sech}^2(t)$, i.e.
\begin{equation} \label{eq:example}
	u''(t) + (2\operatorname{sech}^2(t)-\lambda^2)u(t) = 0,
\end{equation}
where $f(\lambda,t,u)$ is a nonlinear perturbation as in \eqref{eq:f-assumptions}. Consequently, $p(\lambda,t)= 0$ and $q(\lambda,t)= 2\operatorname{sech}^2(t)-\lambda^2$ in \eqref{eq:2nd-order-family-semi} and these coefficients have limits as $t \to \pm\infty$ given by
\begin{equation*}
	p^\pm(\lambda) = 0, \qquad q^\pm(\lambda) = -\lambda^2.
\end{equation*}
The limiting equations are $u'' -\lambda^2 u = 0$ and they satisfy the hyperbolicity condition that was assumed above. By a direct computation,
\begin{equation*}
	u_-(\lambda,t) = e^{\lambda t}(\lambda - \tanh t),\; u_+(\lambda,t) = e^{-\lambda t}(\lambda + \tanh t)
\end{equation*}
are solutions of \eqref{eq:example} such that $u_-(\lambda,t)$ tends to $0$ as $t\rightarrow-\infty$, and $u_+(\lambda,t)$ tends to $0$ as $t\rightarrow\infty$.  Moreover,
\begin{align*}
u_-'(\lambda,t)=e^{\lambda t}\left(\lambda^2-\lambda\tanh t-\operatorname{sech}^2 t\right),\quad u_+'(\lambda,t)=e^{-\lambda t}\left(-\lambda^2-\lambda\tanh t+\operatorname{sech}^2 t\right), 
\end{align*}
and thus we have in \eqref{eq:Evans-parity-D}
\begin{align*}
	\mathcal{L}_D(\lambda)=\det \begin{pmatrix} \lambda & \lambda \\ \lambda^2-1 & 1-\lambda^2 \end{pmatrix} = 2\lambda(1-\lambda^2).
\end{align*}
As $
\mathcal{L}_D(1/2)\cdot \mathcal{L}_D(3/2)<0$,
we see that \eqref{eq:2nd-order-family-semi} has a bifurcation point.

\vspace*{1cm}
\begin{minipage}{1.2\textwidth}
\begin{minipage}{0.4\textwidth}

Robert Skiba, Daniel Strzelecki\\
Faculty of Mathematics\\ and Computer Science\\
Nicolaus Copernicus University in Toru\'n\\
Chopina 12/18\\
87-100 Torun\\
Poland\\
E-mail: robert.skiba@mat.umk.pl\\
E-mail: daniel.strzelecki@mat.umk.pl
\end{minipage}
\hfill
\begin{minipage}{0.6\textwidth}

Nils Waterstraat\\
Martin-Luther-Universit\"at Halle-Wittenberg\\
Naturwissenschaftliche Fakult\"at II\\
Institut f\"ur Mathematik\\
06099 Halle (Saale)\\
Germany\\
nils.waterstraat@mathematik.uni-halle.de
\end{minipage}
\end{minipage}

\end{document}